\documentclass[reqno]{amsart}
\usepackage{enumerate, bbm}
\usepackage{amssymb,url,booktabs}
\usepackage{xcolor}
\usepackage{mathrsfs}

\usepackage[
  colorlinks=true,
  linkcolor=blue,
  citecolor=red,
  urlcolor=black,
  filecolor=black
]{hyperref}

\numberwithin{equation}{section}

\newcommand{\ip}[2]{\langle #1, #2 \rangle}
\newcommand{\be}{\begin{eqnarray}}
\newcommand{\ee}{\end{eqnarray}}
\newcommand{\ce}{\begin{eqnarray*}}
\newcommand{\de}{\end{eqnarray*}}
\newtheorem{theorem}{Theorem}[section]
\newtheorem{lemma}[theorem]{Lemma}
\newtheorem{remark}[theorem]{Remark}
\newtheorem{definition}[theorem]{Definition}
\newtheorem{proposition}[theorem]{Proposition}
\newtheorem{examples}[theorem]{Example}
\newtheorem{corollary}[theorem]{Corollary}
\newtheorem{assumption}[theorem]{Assumption}

\def\var{{\mathrm{var}}}

\def\eps{\varepsilon}

\def\e{\mathrm{e}}

\def\p{\partial}

\def\<{{\langle}}
\def\>{{\rangle}}
\def\bz{{\mathbf{z}}}

\def\tr{\mathrm {tr}}

\def\bb2{{\boldsymbol{2}}}
\def\no{\nonumber}
\def\={&\!\!=\!\!&}

\def\bB{{\mathbf B}}
\def\bC{{\mathbf C}}

\def\cB{{\mathcal B}}
\def\cC{{\mathcal C}}

\def\cF{{\mathcal F}}
\def\cG{{\mathcal G}}
\def\cH{{\mathcal H}}

\def\cK{{\mathcal K}}
\def\cL{{\mathcal L}}
\def\cM{{\mathcal M}}

\def\cP{{\mathcal P}}

\def\cR{{\mathcal R}}
\def\cS{{\mathcal S}}

\def\mE{{\mathbb E}}

\def\mK{{\mathbb K}}
\def\mH{{\mathbb H}}

\def\mN{{\mathbb N}}

\def\mP{{\mathbb P}}

\def\mR{{\mathbb R}}

\def\bP{{\mathbf P}}
\def\bQ{{\mathbf Q}}

\def\b1{{\mathbbm 1}}

\def\sA{{\mathscr A}}
\def\sB{{\mathscr B}}

\def\sF{{\mathscr F}}

\def\sK{{\mathscr K}}

\def\E{\mathbb E}

\def\geq{\geqslant}
\def\leq{\leqslant}
\def\ge{\geqslant}
\def\le{\leqslant}

\def\div{\mathord{{\rm div}}}

\def\var{{\mathrm{var}}}

\def\eps{\varepsilon}

\def\e{\mathrm{e}}

\def\p{\partial}

\def\<{{\langle}}
\def\>{{\rangle}}
\def\bz{{\mathbf{z}}}

\def\tr{\mathrm {tr}}

\def\dif{{\mathord{{\rm d}}}}

\def\no{\nonumber}
\def\={&\!\!=\!\!&}
\def\bt{\begin{theorem}}
\def\et{\end{theorem}}
\def\bl{\begin{lemma}}
\def\el{\end{lemma}}
\def\br{\begin{remark}}
\def\er{\end{remark}}
\def\bd{\begin{definition}}
\def\ed{\end{definition}}
\def\bp{\begin{proposition}}
\def\ep{\end{proposition}}
\def\bc{\begin{corollary}}
\def\ec{\end{corollary}}

\def\geq{\geqslant}
\def\leq{\leqslant}
\def\ge{\geqslant}
\def\le{\leqslant}

\def\div{\mathord{{\rm div}}}

 \def\R{\mathbb R}
 \def\R{\mathbb R}    
\def\N{\mathbb N}  
   
\def\<{\langle} \def\>{\rangle}

\newcommand{\norm}[1]{\left\lVert #1 \right\rVert}
\newcommand{\abs}[1]{\left\lvert #1 \right\rvert}
\newcommand{\inner}[2]{\left\langle #1, #2 \right\rangle}

\allowdisplaybreaks

\begin{document}

\title[Propagation of chaos for kinetic McKean--Vlasov SDEs]{Quantitative Propagation of Chaos and Fluctuations
for Kinetic McKean--Vlasov SDEs with Singular Interaction Kernels}

\author{Zimo Hao, Xicheng Zhang and Xianliang Zhao}

\hypersetup{
  pdftitle={Quantitative Propagation of Chaos and Fluctuations for Kinetic McKean--Vlasov SDEs with Singular Interaction Kernels},
  pdfauthor={Zimo Hao, Xicheng Zhang and Xianliang Zhao}
}

\address{Zimo Hao: School of Mathematics and Statistics, Beijing Institute of Technology, Beijing 100081, China,
Email: zmhao@bit.edu.cn}

\address{Xicheng Zhang:
School of Mathematics and Statistics, Beijing Institute of Technology, Beijing 100081, China\\
Faculty of Computational Mathematics and Cybernetics, Shenzhen MSU-BIT University, 518172 Shenzhen, China\\
Email: XichengZhang@gmail.com
 }

\address{Xianliang Zhao:
{Beijing International Center for Mathematical Research, Peking University, Beijing 100871, China\\
Email: xzhaomath@gmail.com}
 }

\begin{abstract}
We prove a quantitative propagation of chaos estimate and a central limit theorem
for the particle system associated with
a class of degenerate kinetic McKean--Vlasov SDEs with external drifts and singular interaction kernels in Kato's class. In particular, the interaction kernel can be in 
the mixed $L^q_tL^{p_v}_vL^{p_x}_x$-space, where $\frac2q+\frac{3d}{p_x}+\frac d{p_v}<1$. 
For the associated $N$-particle system, we obtain a path-space relative entropy bound of order $k/N$ for the first $k$ particles, assuming only entropic chaoticity of the initial data. The key ingredients are kinetic Krylov--Khasminskii estimates and a conditional Hilbert-space subgaussian estimate for empirical interaction fields. For the CLT, we also prove
a Berry--Esseen-type bound
for finite-dimensional projections.


\end{abstract}
\maketitle


\section{Introduction}

\subsection{Background and main contributions}
The propagation of chaos problem, originating from the works of Kac and McKean, asks whether a finite number of particles in a large interacting system become asymptotically independent and distributed according to a nonlinear limiting equation; see \cite{Kac56,McKean67,Sznitman91,Meleard96,ChaintronDiez22a,ChaintronDiez22b} and the references therein. In this paper we use the relative entropy method to study this problem for a class of degenerate, hypoelliptic McKean--Vlasov dynamics with singular drifts and singular interaction kernels. More precisely, we consider the kinetic McKean--Vlasov SDE
\begin{align}\label{kMVSDE}
    \begin{cases}
        \dif X_t=V_t\dif t,\\
        \dif V_t=b(t,Z_t)\dif t+(K*\mu)(t,Z_t)\dif t+\sqrt{2}\dif W_t,
    \end{cases}
\end{align}
where $Z_t=(X_t,V_t)\in\mR^{2d}$, $\mu=(\mu_t)_{t\geq 0}$ is the law of $Z$, and $W$ is a standard $d$-dimensional Brownian motion. For a measure flow $\mu=(\mu_t)_{t\ge0}$, we use the convention
\begin{align*}
    (K*\mu)(t,z):=\int_{\mR^{2d}}K(t,z-z')\mu_t(\dif z'),
    \qquad z\in\mR^{2d}.
\end{align*}

The main novelties of this paper are threefold.
\begin{enumerate}[(i)]
    \item \emph{Kato-class singular interaction.} The interaction kernel is allowed to be singular on the full kinetic phase space $z=(x,v)$: it need not be bounded, continuous, a gradient, repulsive, or generated by a potential. In particular, the assumptions include genuinely unbounded and discontinuous mixed-Lebesgue examples, and are matched to the Krylov--Khasminskii estimates available for degenerate kinetic SDEs.
    \item \emph{Path-space entropy rate $k/N$.} For exchangeable initial data we prove a path-space relative entropy estimate with rate $k/N$ for the first $k$ particle paths. In the i.i.d. case this yields the normalized full $N$-particle entropy rate $1/N$, and in general the initial discrepancy is kept explicitly through $\cH(\mu_0^{N,N}|\mu_0^{\otimes N})$. Thus the estimate is an entropy-per-particle bound on trajectories, not merely a fixed-time weak convergence statement.
    \item \emph{Kinetic fluctuation CLT and Berry--Esseen bounds.} In a fluctuation regime with additional anisotropic Besov regularity, the empirical field $\sqrt N(\mu_t^N-\mu_t)$ converges to a Gaussian linearized kinetic process in weighted negative Besov spaces, and finite-dimensional projections satisfy Berry--Esseen-type estimates. The limiting object is identified without imposing a gradient, repulsive, or potential structure on $K$, even on finite time horizons.
\end{enumerate}

Kinetic McKean--Vlasov equations are closely related to Vlasov--Fokker--Planck equations and Langevin dynamics with mean-field interaction; see, for instance, \cite{CattiauxGuillinMonmarcheZhang19,GuillinLiuWuZhang21} for related entropy and hypocoercive approaches. Propagation of chaos and law of large numbers for kinetic systems with bounded or sufficiently regular interactions have also been studied from several viewpoints. Jabin and Wang \cite{JW16} treat essentially the same kinetic structure as the one considered here, without the external drift $b$ and with bounded interaction kernel $K$. Bellingeri and Coppini \cite{BellingeriCoppini26} prove a law of large numbers on finite time horizons, while Chen, Lin, Ren and Wang \cite{ChenLinRenWang24} establish uniform-in-time propagation of chaos for kinetic mean-field Langevin dynamics under convexity-type assumptions.

Compared with uniformly elliptic mean-field diffusions, \eqref{kMVSDE} is degenerate: the noise acts only on the velocity component, while the position component is regularized only through the transport relation $\dif X_t=V_t\dif t$. This anisotropic smoothing is reflected in the Kolmogorov heat kernel and in the kinetic Kato classes used below; see \cite{Ko34,H,CMPZ,RZ25}. A second difficulty is that both the external drift and the interaction kernel may be singular. Singular SDEs with Kato-type drifts have been studied in non-degenerate and degenerate settings; see, among others, \cite{BassChen03,KR,ZZ18,XZ20,RZ25}. For propagation of chaos with singular interactions, entropy methods have proved particularly robust; see \cite{JabinWang18,Lacker21,Cattiaux24,Han22,HaoRocknerZhang24,BreschJabinWang19,BreschJabinWang23,Wang26}. Related kinetic singular moderately interacting particle systems have also been investigated recently in \cite{HaoJabirMenozziRocknerZhang24}.

A closely related but technically different line of work is the BBGKY-hierarchy approach of Bresch, Jabin and Soler \cite{BJS25} for kinetic Vlasov--Fokker--Planck equations with singular repulsive forces. Their method reaches Coulomb/Poisson-type singularities under structural and a priori assumptions adapted to that setting, including a potential form of the interaction and suitable hierarchy bounds. Rather than using a potential/hierarchy structure to access the Coulomb--Poisson regime, the path-space entropy argument is aimed at a different result range: flexible Kato-class kernels on the full kinetic phase space, including genuinely unbounded and discontinuous mixed-Lebesgue examples. In this range it gives, over every finite time interval, explicit entropy rates and second-order fluctuation estimates, while imposing no density, smoothness, moment, or weighted-energy condition on the initial $N$-particle law beyond exchangeability.

\subsection{Main results}
For $\beta\geq 0$, let $\mK^d_\beta$ be the kinetic Kato class introduced in Definition \ref{def:Kato} below.
We impose the following assumption throughout the paper.
\begin{assumption}\label{ass:main}
Let $b=b_0+b_1$, where $|b_1|,|K|\in\mK^d_1$, $|b_1|^2,|K|^2\in\mK^d_0$, and 
for some $\kappa_0>0$,
    \begin{align*}
        |b_0(t,0)|\le\kappa_0,
        \qquad
        |\nabla_z b_0(t,z)|\le\kappa_0,\ \ (t,z)\in\mR_+\times\mR^{2d}.
    \end{align*}
\end{assumption}

Note that for a probability measure flow $\mu=(\mu_s)_{s\geq 0}$, by Fubini's theorem and translation invariance of the Kato norm (see Lemma \ref{lem:Kato_con} below),
\begin{align*}
    \cK^{(1)}_\lambda(|b_1+K*\mu|;\delta)
    &\le \cK^{(1)}_\lambda(|b_1|;\delta)+\cK^{(1)}_\lambda(|K|;\delta),\\
    \cK^{(0)}_\lambda(|b_1+K*\mu|^2;\delta)
    &\le
    2\cK^{(0)}_\lambda(|b_1|^2;\delta)
    +2\cK^{(0)}_\lambda(|K|^2;\delta).
\end{align*}
In particular, the frozen drifts of the form $b+K*\mu$ satisfy the same linear Kato assumptions.

\medskip\noindent\textbf{Mixed Lebesgue sufficient condition.} Assume that, for some $q,p_x,p_v\ge2$,
\begin{align}\label{assumption:K}
    b_1,K\in L^q_t\big(\mR_+;L^{p_v}_vL^{p_x}_x\big),
    \quad
    \tfrac2q+\tfrac{3d}{p_x}+\tfrac d{p_v}<1.
\end{align}
Then $b_1,K\in\mK_1$, $|b_1|^2,|K|^2\in\mK_0$ (see Example \ref{ex:Lp} below). 

We first introduce the following notion of weak solution.
\bd\label{def:SDE}
Let $\mathfrak U:=(\Omega,\sF,\mP, (\sF_t)_{t\geq 0})$ be a stochastic basis, let $Z=(X,V)$ be a continuous $\sF_t$-adapted process with values in $\mR^{2d}$, and let $W$ be a continuous $\sF_t$-adapted process with values in $\mR^d$.
Let $\mu_0\in\cP(\mR^{2d})$.
We call $(Z,W,\mathfrak U)=(X,V,W,\mathfrak U)$ a weak solution of kinetic SDE \eqref{kMVSDE} with initial distribution $\mu_0$ if
\begin{enumerate}[(i)]
\item $\mu_0=\mP\circ Z^{-1}_0$ and $W$ is a standard $d$-dimensional Brownian motion on $\mathfrak U$.

\item For each $t>0$,
$\int^t_0\bigl(|b(s,Z_s)|+|(K*\mu)(s,Z_s)|\bigr)\dif s<\infty$, a.s., and
$$
X_t=X_0+\int^t_0V_s\dif s,\ \ V_t=V_0+\int^t_0[b+K*\mu](s, Z_s)\dif s+\sqrt 2W_t,\ \ \text{a.s.},
$$
where $\mu_t=\mP\circ Z_t^{-1}$ is the time marginal distribution of the solution $Z$.
\end{enumerate}
\ed

Our first result is the corresponding weak well-posedness of the nonlinear equation \eqref{kMVSDE}.

\begin{theorem}\label{thm:well-posed}
    Under Assumption \ref{ass:main}, for every initial law $\mu_0\in\cP(\mR^{2d})$, the kinetic McKean--Vlasov SDE \eqref{kMVSDE} admits a unique weak solution in the sense of Definition \ref{def:SDE}.
\end{theorem}

\br\rm
When $K=0$, the well-posedness of the linear kinetic SDE \eqref{kMVSDE} was established in \cite{RZ25}.
For $K$ in Kato's class, the proof of Theorem \ref{thm:well-posed} relies on a stability estimate in relative entropy. If two frozen kinetic equations have drifts differing by a Kato-class perturbation, Girsanov's theorem and the Krylov estimate yield a path-space entropy bound. \er

Let $N\in\mN$ and $(W^i)_{i=1}^N$ be independent standard $d$-dimensional Brownian motions.
The $N$-particle system associated with \eqref{kMVSDE} is
\begin{align}\label{IPS}
    \begin{cases}
        \dif X_t^{N,i}=V_t^{N,i}\dif t, \qquad i=1,\ldots,N,\\
        \dif V_t^{N,i}=b(t,Z_t^{N,i})\dif t+(K*\mu^N_t)(t,Z_t^{N,i})\dif t+
        \sqrt 2\dif W_t^i,
    \end{cases}
\end{align}
where $Z_t^{N,i}:=(X_t^{N,i},V_t^{N,i})$ and $\mu^N_t$ is the empirical measure associated with $(Z_t^{N,i})_{i=1,\cdots,N}$, i.e.,
$$
\mu^N_t:=\frac1N\sum_{j=1}^N\delta_{Z_t^{N,j}}. 
$$
The above IPS can be regarded as a linear kinetic equation on $(\mR^{2d})^N$.
In this case, for $\bz=(z_1,\cdots, z_N)\in(\mR^{2d})^N$, the drift is given by
$$
B_i^N(t,z_1,\ldots,z_N)
    :=b(t,z_i)+\frac1N\sum_{j=1}^N K(t,z_i-z_j)\in\mR^d,
    \quad i=1,\ldots,N.
$$
Throughout the sequel we assume
$$
K(t,0)=0.
$$
The product drift consists of the Lipschitz--linear part inherited from $b_0$
and the singular part generated by $b_1$ and the empirical interaction. By
Assumption~\ref{ass:main}, Lemma~\ref{lem:Kato_con}, and the definition of the
product kinetic Kato norm, this singular part belongs to the corresponding
product Kato class:
$$
|b_1(t,z_i)+N^{-1}\sum_{j=1}^N K(t,z_i-z_j)|\in \mK^{Nd}_1,
\qquad
|b_1(t,z_i)+N^{-1}\sum_{j=1}^N K(t,z_i-z_j)|^2\in \mK^{Nd}_0,
$$
for each $i=1,\ldots,N$. Thus, 
by \cite[Theorem 4.2]{RZ25}, \eqref{IPS} is weakly well-posed for every initial distribution $\mu_0^{N,N}\in\cP((\mR^{2d})^N)$.

We now state the quantitative propagation of chaos result. Let $\bP_{[0,t]}^{N,k}$ be the law of the first $k$ particle paths on $[0,t]$, and let $\bP_{[0,t]}$ be the law of the nonlinear process on the same time interval:
\begin{align}\label{path-laws}
    \bP_{[0,t]}^{N,k}:=\cL(Z_{[0,t]}^{N,1},\ldots,Z_{[0,t]}^{N,k}),
    \qquad
    \bP_{[0,t]}:=\cL(Z_{[0,t]}).
\end{align}
We also write
\begin{align}\label{time-laws}
    \mu_t^{N,k}:=\cL(Z_t^{N,1},\ldots,Z_t^{N,k}),
    \qquad
    \mu_t:=\cL(Z_t).
\end{align}

\begin{theorem}\label{thm:main}
    Under Assumption \ref{ass:main}, assume that $\mu_0^{N,N}$ is exchangeable. Then, for every $t>0$, there exists a constant $C_t>0$, independent of $N$, $k$ and $\mu_0^{N,N}$, such that for all $N\in\mN$ and $k=1,\ldots,N$,
    \begin{align}\label{PoC:main}
        \cH(\bP_{[0,t]}^{N,k}|\bP_{[0,t]}^{\otimes k})
        \le
        \frac{C_t k}{N}
        \left(\cH(\mu_0^{N,N}|\mu_0^{\otimes N})+1\right).
    \end{align}
\end{theorem}

Theorem \ref{thm:main} is a path-space relative entropy estimate. In particular, if the initial law is chaotic in the entropic sense, namely
\begin{align*}
    \frac1N\cH(\mu_0^{N,N}|\mu_0^{\otimes N})\longrightarrow0,
\end{align*}
then, for each fixed $k$ and $t>0$,
\begin{align*}
    \cH(\bP_{[0,t]}^{N,k}|\bP_{[0,t]}^{\otimes k})\longrightarrow0,
\end{align*}
and hence $\bP_{[0,t]}^{N,k}$ converges to $\bP_{[0,t]}^{\otimes k}$ in total variation by Pinsker's inequality. Taking $k=N$ in \eqref{PoC:main} also gives the normalized full path-space entropy estimate
\begin{align*}
    \frac1N\cH(\bP_{[0,t]}^{N,N}|\bP_{[0,t]}^{\otimes N})
    \le
    C_t\left(
        \frac1N\cH(\mu_0^{N,N}|\mu_0^{\otimes N})
        +\frac1N
    \right).
\end{align*}
In particular, if $\mu_0^{N,N}=\mu_0^{\otimes N}$, the normalized full $N$-particle entropy is of order $1/N$, and the $k$-marginal estimate gives the order $k/N$ in relative entropy. The $1/N$ scale for the normalized full entropy is the natural entropy-per-particle order for mean-field systems.

Let us briefly explain the proof of Theorem \ref{thm:main}. By Girsanov's theorem, the path-space entropy is controlled by the squared fluctuation of the empirical interaction field, and the variational formula reduces the problem to a uniform exponential estimate of the form
\begin{align*}
    \sup_{N\ge1}
    \mE\exp\left(
        \frac1N\int_0^t
        \left|
            \sum_{i=1}^N
            \left(K(s,Z_s^1-Z_s^i)-(K*\mu_s)(s,Z_s^1)\right)
        \right|^2
        \dif s
    \right)<\infty
\end{align*}
for small times, where $(Z^i)_{i=1}^N$ are independent copies of $Z$. Since $K$ is singular, standard bounded-difference or Lipschitz concentration is unavailable. We instead regard the fluctuation as an empirical mean of conditionally independent $L^2([0,t];\mR^d)$-valued variables. Kato--Krylov and Khasminskii estimates give conditional square-exponential integrability, and a Hilbert-space subgaussian estimate yields the required bound. The short-time estimate is then extended by the Markov property and entropy chain rule. The concentration argument used in the present paper was motivated by the
martingale-difference estimate developed by Wang and Zhao
\cite[Lemmas~3.3--3.4 and Remark~3.2]{WangZhao26} (see also \cite{JabinWang18, WangZhaoZhu23}), and is extended here to
second-order kinetic systems with singular Kato-class interactions. Related applications of sub-Gaussian-type concentration estimates to singular
interacting particle systems are being investigated in the ongoing-work of
Galeati, Mayorcas and Weinberger
\cite{GaleatiMayorcasWeinberger26}.

\begin{remark}\label{rem:examples-main}\rm
    Theorem \ref{thm:main} gives concrete consequences for singular $L^p$ kernels.
    \begin{enumerate}[(i)]
        \item If $K=K(v)$, $b_1=0$ and $b_0(v)=-\nabla U(v)$ for a smooth potential $U$, then the velocity component of \eqref{kMVSDE} solves
        \begin{align*}
            \dif V_t=-\nabla U(V_t)\dif t+(K*\cL(V_t))(V_t)\dif t+\sqrt{2}\dif W_t.
        \end{align*}
        In this case, the condition \eqref{assumption:K} allows any $K\in L^p(\mR^d)$ with $p>d$, and Theorem \ref{thm:main} yields
        \begin{align*}
            \cH\big(\cL(V_t^{N,1},\ldots,V_t^{N,k})\,\big|\,\cL(V_t)^{\otimes k}\big)
            \le
            \frac{C_t k}{N}
            \left(
                \cH\big(\cL(V_0^{N,1},\ldots,V_0^{N,N})\,\big|\,\cL(V_0)^{\otimes N}\big)+1
            \right).
        \end{align*}
       When $b=0$, \cite{HaoRocknerZhang24} obtained the propagation of chaos for any $K\in L^p(\mR^d)$ with $p>d$, but without any rate. By contrast, we establish a path-space relative entropy estimate with the explicit rate $k/N$.
        \item If $K=H(x)/|x|^\gamma$ with $\gamma\in[0,1/3)$ and $H\in L^p(\mR^d)$ with $p\in(3d/(1-3\gamma),\infty]$, 
        then it is easy to see that (see \cite[Remark 2.4]{RZ25})
        $$
        |K|\in \mK^d_1,\ \ |K|^2\in \mK^d_0.
        $$
        This extends the bounded-kernel relative entropy estimate of \cite{JW16} to
	a class of unbounded $L^p$ interaction kernels in the kinetic setting. Moreover, unlike \cite{JW16}, our result requires only that the initial $N$-particle law $\mu_0^{N,N}$ is exchangeable, with no additional regularity assumption.
    \end{enumerate}
\end{remark}

Beyond the law of large numbers estimate in Theorem~\ref{thm:main}, we also obtain a second-order fluctuation result. In Section~\ref{sec:fluctuation-consequence} we work under the additional fluctuation-regime assumptions that the external drift is absent and the interaction kernel is time-independent, namely $b\equiv0$ and $K(t,z)=K(z)$. If, moreover, $K$ belongs to a kinetic Kato class $\mK_\beta$, $\beta\in(1,2)$, with the compatible anisotropic Besov regularity required in Assumption~\ref{ass:clt}, and if the initial law has enough positive anisotropic Besov regularity, then for i.i.d. initial data the empirical fluctuation field
\[
    \eta_t^N:=\sqrt N(\mu_t^N-\mu_t)
\]
converges to a Gaussian linearized fluctuation process $\eta$; see Theorem \ref{thm:clt} below. More precisely, for every finite horizon $T>0$, every admissible $\alpha<-2d$, and every $\ell<0$, the convergence holds in the weighted negative kinetic Besov space
\[
    C([0,T];\bB^{\alpha-3,2}_{2;a,\ell-1}(\mR^{2d})).
\]
The limit is the unique solution of the linearized martingale problem around the nonlinear law flow $(\mu_t)_{t\in[0,T]}$; in formal terms it solves
\[
    \partial_t\eta_t
    =\Delta_v\eta_t-v\cdot\nabla_x\eta_t
    -\div_v\big((K*\mu_t)\eta_t+(K*\eta_t)\mu_t\big)
    -\sqrt{2}\,\div_v(\sqrt{\mu_t}\,\xi),
\]
where $\xi$ is a $\mR^d$-valued white noise on $\mR_+\times\mR^{2d}$ with independent components. 
Thus the limiting object is a generalized Ornstein--Uhlenbeck process: its initial condition is the usual empirical-measure Gaussian field with covariance
\[
    \mE\big[\langle\eta_0,\varphi\rangle\langle\eta_0,\psi\rangle\big]
    =\langle\mu_0,\varphi\psi\rangle
     -\langle\mu_0,\varphi\rangle\langle\mu_0,\psi\rangle,
\]
and the driving martingale field is independent of $\eta_0$ and has covariance
\[
    \mE[M_t^\varphi M_s^\psi]
    =2\int_0^{t\wedge s}\langle\mu_r,\nabla_v\varphi\cdot\nabla_v\psi\rangle\,\dif r.
\]

Classical fluctuation CLTs for McKean--Vlasov systems are typically proved under bounded or sufficiently smooth interaction coefficients; see, for instance, \cite{Sznitman91,Meleard96}. 
The closest technical point of comparison for the fluctuation analysis is the martingale strategy of \cite{WangZhaoZhu23}; the new difficulty here is the degenerate kinetic structure, where noise acts only on the velocity variable and compactness must be formulated in anisotropic, weighted, negative Besov spaces on the full phase space. 

\vspace{2mm}

For finite-dimensional projections of $\eta_t^N$, we further prove a CLT rate along with a Berry--Esseen estimate against the corresponding Gaussian limit. More precisely, 
for any $m\in\mN$ and $\varphi_1,\ldots,\varphi_m\in C_c^\infty(\mR^{2d})$, we consider 
\[
Y_t^N:=\bigl(\langle\eta_t^N,\varphi_1\rangle,\ldots,
\langle\eta_t^N,\varphi_m\rangle\bigr),\qquad
Y_t:=\bigl(\langle\eta_t,\varphi_1\rangle,\ldots,
\langle\eta_t,\varphi_m\rangle\bigr).
\]
Theorem~\ref{thm:smooth-weak-rate} below gives the quantitative estimate
\begin{align*}
\sup_{t\in[0,T]}
\bigl|\mE F(Y_t^N)-\mE F(Y_t)\bigr|
\lesssim_C N^{-1/2}
\left(\|\nabla F\|_{L^\infty(\mR^m)}+
\|\nabla^2F\|_{L^\infty(\mR^m)}\right).
\end{align*}
By smoothing indicators, Theorem~\ref{thm:berry-esseen} further yields, for each fixed $t\in[0,T]$ such that the coordinate variances of $Y_t$ are positive, the following bound, in which inequalities between vectors are understood componentwise (see \eqref{Berry-Esseen-m}):
\begin{align*}
\sup_{x\in\mR^m}
\left|\mathbb P(Y_t^N\le x)-\mathbb P(Y_t\le x)\right|
\le C N^{-1/6},
\end{align*}

Quantitative CLTs for empirical-measure fluctuations have also been established in several complementary settings. For classical Vlasov dynamics with a smooth long-range potential, Duerinckx \cite{Duerinckx21} obtained the optimal scalar rate $N^{-1/2}$ in both the $1$-Wasserstein and Kolmogorov distances, with constants that grow in time. Bernou and Duerinckx \cite{BernouDuerinckx26} subsequently proved uniform-in-time quantitative CLTs for weakly interacting Brownian and kinetic Langevin particles with smooth coefficients. Their general scalar estimate is of order $N^{-1/3}$ in the second-order Zolotarev distance and improves to $N^{-1/2}$ under a non-degeneracy condition. For first-order diffusions on the torus, assuming bounded interaction kernels with summable Fourier modes and sufficiently large diffusion, Xie \cite{Xie25} established a uniform-in-time scalar Kolmogorov bound of order $N^{-1/7}$ after normalization by the finite-$N$ variance. 
Our contribution is therefore twofold. First, the path-space entropy estimate of Theorem~\ref{thm:main} is strong enough to control the fluctuation field at the scale $\sqrt N$ even for unbounded and discontinuous Kato-class interactions. Second, the limiting Gaussian process is identified for genuinely kinetic singular interactions without assuming a gradient, repulsive, or potential structure of $K$. Together with the Berry--Esseen estimate above, this gives a quantitative Gaussian fluctuation result with an explicit linearized kinetic limit for flexible Kato-class interactions on the full kinetic phase space.

\subsection{Structure and roadmap}
Section~\ref{sec:pre} collects the analytic and probabilistic preliminaries: anisotropic Besov spaces, kinetic Kato classes, Krylov--Khasminskii estimates, and the relative entropy formula used throughout the Girsanov arguments.

Section~\ref{sec:subgaussian} proves the Hilbert-space subgaussian estimate that controls empirical interaction fields without boundedness or Lipschitz assumptions. This is the key concentration input for both the entropy estimate and the fluctuation analysis.

Section~\ref{sec:well} proves weak well-posedness by a freezing-and-Picard argument. It then establishes Theorem~\ref{thm:main} by combining Girsanov's theorem, the subgaussian estimate, Khasminskii's bound, and the entropy chain rule.

Section~\ref{sec:fluctuation-consequence} proves Theorem~\ref{thm:clt}. Uniform $L^2$ bounds yield tightness in weighted negative Besov spaces; subsequential limits are identified through a stable martingale central limit theorem and a paraproduct-based uniqueness argument; the final subsection derives Berry--Esseen-type bounds for finite-dimensional projections.

\subsection*{Notation and conventions}
\begin{itemize}
\item
Throughout the paper, $d\ge1$ is the spatial dimension. We write $z=(x,v)\in\mR^d\times\mR^d=\mR^{2d}$ for a point in the phase space, where $x$ is the position and $v$ is the velocity. 

\item $\cP(E)$ is the space of probability measures on a measurable space $E$, and $\cM(E)$ the space of finite signed measures. 
\item $\|\mu\|_{\var}$ denotes the total variation norm. 
\item For a process $Z$, $\cL(Z)$ is its law, and $\cL(Z_{[0,t]})$ is the path law on the time interval $[0,t]$. 
\item The relative entropy between probability measures $\mu$ and $\nu$ is denoted by $\cH(\mu|\nu)$. For $k,N\in\mN$ with $k\le N$, $\mu_t^{N,k}$ is the $k$-particle time-marginal law.
\item Constants denoted by $C$, $c$, $C_T$, etc.~may change from line to line. 
We also use $A\lesssim B$ or $A\lesssim_C B$ to denote $A\leq CB$ for some unimportant constant, where
the dependence of the constant $C$ on parameters or functions can be traced from the context.

\end{itemize}

\section{Preliminaries}\label{sec:pre}

In this section we recall the definitions of anisotropic Besov spaces and Kato's classes associated with the 
Brownian kinetic operator,
as well as relative entropy and its basic properties. 
Moreover, we prove a crucial Khasminskii estimate for linearized kinetic SDEs by Girsanov's theorem.

\subsection{Anisotropic Besov spaces and Kato's class}\label{subsec:anisotropic-besov}
Let $p_t(x,v)$ be the joint density of the process $\left(\sqrt 2\int_0^t W_s\dif s,\sqrt 2W_t\right)$. A direct computation 
gives (see \cite{Ko34})
\begin{align}\label{density}
    p_t(z)=p_t(x,v)
    =\left(\frac{\sqrt3}{2\pi t^2}\right)^d
      \exp\left(-\frac{3|2x-tv|^2+t^2|v|^2}{4t^3}\right).
\end{align}
The semigroup associated with $p_t(z)$ is defined by
$$
P_tf(z):=\mE f\left(x+tv+\sqrt 2\int_0^t W_s\dif s, v+\sqrt 2W_t\right)
=\int_{\mR^{2d}}f(x+tv+y, v+w)p_t(y,w)\dif w\dif y.
$$
By It\^o's formula, it is easy to see that
$$
\p_tP_tf=(\Delta_v+v\cdot\nabla_x)P_t f:=\sK P_t f.
$$
By definition, we also have
$$
p_t(x,v)=t^{-2d}p_1(t^{-3/2}x, t^{-1/2}v).
$$
In particular, the anisotropy scaling is
$$
a=(3,1).
$$
We introduce the anisotropic distance
$$
|z|_a=|x|^{1/3}+|v|.
$$
For any $r>0$ and $z=(x,v)\in\mR^{2d}$, the ball with respect to $|\cdot|_a$ is defined by
$$
B_r^a(z):=\{z':|z-z'|_a\le r\},\ B_r^a:=B_r^a(0).
$$
Let $\chi\in C^\infty_c(\mathbb{R}^{2d})$ be symmetric, nonnegative, equal to $1$ on $B_1^a$ and $0$ outside $B_{4/3}^a$. For $j\ge-1$, define a function on $\mR^{2d}$
\begin{align*}
\phi_j^a(\xi)=
\begin{cases}
\chi(\xi), & j=-1,\\
\chi(2^{-a(j+1)}\xi)-\chi(2^{-aj}\xi), & j\ge0,
\end{cases}
\end{align*}
where for $\xi=(\xi_1,\xi_2)\in\mR^d\times\mR^d=\mR^{2d}$,
$$
2^{-aj}\xi:=(2^{-3j}\xi_1,2^{-j}\xi_2).
$$
In particular, 
$$
\phi_j^a(\xi)=\phi_0^a(2^{-aj}\xi),\ \ \sum_{j\ge-1}\phi_j^a(\xi)=1.
$$

Let $\cS$ be the space of all Schwartz functions on $\mR^{2d}$ and $\cS'$ be its topological dual, i.e. the space of  tempered distributions. We use $\mathcal{F}$ and $\mathcal{F}^{-1}$ to denote the Fourier transform and its inverse on $\cS'$. Then we define block operators for $f\in\mathcal{S}'$ by
\begin{align}\label{RR1}
\mathcal{R}_j^a f:=\mathcal{F}^{-1}(\phi_j^a\mathcal{F}f)=(\cF^{-1}\phi_j^a)\ast f.
\end{align}

\begin{definition}[Weighted Besov spaces \cite{HZZZ22}]
For $\alpha\in\mathbb{R}$ and $p,q\in[1,\infty]$, define
\begin{align*}
\bB_{p;a}^{\alpha,q}:=\Big\{f\in\mathcal{S}':\ \|f\|_{\bB_{p;a}^{\alpha,q}}:=\Big(\sum_{j\ge-1}2^{\alpha j q}\|\mathcal{R}_j^a f\|^q_{L^p}\Big)^{1/q}<\infty\Big\},
\end{align*}
with the usual modification when $q=\infty$.
Moreover, for $\ell\in\mR$, we also define weighted Besov spaces with the norm
$$
    \|f\|_{\bB^{\alpha,q}_{p;a,\ell}}
    :=
    \|\langle z\rangle^{\ell}f\|_{\bB^{\alpha,q}_{p;a}},
    \qquad
    \langle z\rangle=(1+|z|^2)^{1/2}.
$$
For $p=q=\infty$, we write
$$
\bC^\alpha_a:=\bB^{\alpha,\infty}_{\infty;a}.
$$
\end{definition}
By definition,
it is easy to see that for any $\alpha\in\mR$, $p\in[1,\infty]$, $1\le q_1\le q_2\le \infty$, and $\eps>0$,
\begin{align}\label{easy_emb}
    \bB^{\alpha,q_1}_{p;a}\subset \bB^{\alpha,q_2}_{p;a}\subset \bB^{\alpha-\eps,1}_{p;a}.
\end{align}

\bl[Bernstein's inequality]\label{Le22}
For any $m,n\in\mN_0$ and $1\le p\le q\le \infty$,
there is a constant $C=C(d,m,n,p,q)>0$ such that for all $j\geq -1$,
\begin{align*}
    \|\nabla^m_x\nabla^n_v\cR_j^a f\|_{L^{q}}\lesssim_C 2^{j\beta}\|\cR_j^a f\|_{L^p},
\end{align*}
where $\beta:=3m+n+4d(1/p-1/q)$.
In particular, for any $\alpha\in\mR$,
\begin{align}\label{eq:clt-lp-gradient}
    \|\nabla^m_x\nabla^n_v f\|_{\bB^{\alpha,\infty}_{q;a}}
    \lesssim_C
    \|f\|_{\bB^{\alpha+\beta,\infty}_{p;a}},
\end{align}
and for any finite signed measure $\mu\in\cM(\mR^{2d})$,
\begin{align}\label{measure-emb}
    \|\mu\|_{\bB^{4d(1/p-1),\infty}_{p;a}}\lesssim \|\mu\|_{\var}.
\end{align}
\el

The convolution estimate \eqref{Besov-con} below
and the positive-regularity product estimate \eqref{product} below can be found in \cite[Lemma 2.7]{HRZ}. The last estimate \eqref{product-negative} below follows from the standard paraproduct argument; see, for instance, \cite[Lemma 2.11]{HZZZ22}.
\bl\label{Le23}
Let $\alpha,\beta\in\mR$ and $p_1,p_2,p, q_1,q_2,q\in[1,\infty]$ with 
$$
1+1/p=1/p_1+1/p_2,\quad 1/q=1/q_1+1/q_2.
$$ 
It holds that 
\begin{align}\label{Besov-con}
    \|f*g\|_{\bB^{\alpha+\beta,q}_{p;a}}\lesssim \|f\|_{\bB^{\alpha,q_1}_{p_1;a}}\|g\|_{\bB^{\beta,q_2}_{p_2;a}},
\end{align}
and for any $\alpha>0$,
\begin{align}\label{product}
    \|fg\|_{\bB^{\alpha,\infty}_{q;a}}\lesssim \|f\|_{\bB^{\alpha,\infty}_{q_1;a}}\|g\|_{L^{q_2}}+\|f\|_{L^{q_1}}\|g\|_{\bB^{\alpha,\infty}_{q_2;a}},
\end{align}
and if $\alpha+\beta>0$ and $r\in[1,\infty]$ satisfies $1/r=1/p+1/q$,
\begin{align}\label{product-negative}
\|fg\|_{\bB^{\alpha\wedge\beta,\infty}_{r;a}}\lesssim \|f\|_{\bB^{\beta,\infty}_{p;a}}\|g\|_{\bB^{\alpha,\infty}_{q;a}}.
\end{align}
\el
The following lemma comes from \cite[Lemma 2.14]{HRZ}.
\begin{lemma}[Kinetic semigroup estimates]\label{lem:kinetic-besov-semigroup}
For any $p,q\in[1,\infty]$ and $\alpha\in\mR$, we have
$$
    \|P_t f\|_{\bB^{\alpha,q}_{p;a}}
    \lesssim\|f\|_{\bB^{\alpha,q}_{p;a}},\ \ t>0,
$$
and for $\alpha<\beta \in\mR$,
$$
    \|P_t f\|_{\bB^{\beta,1}_{p;a}}
    \lesssim
 (1\wedge t)^{-\frac{\beta-\alpha}2}\|f\|_{\bB^{\alpha,\infty}_{p;a}},\ \ t>0.
$$
\end{lemma}

Next we introduce the Kato class associated with $p_t(x,v)$.
For $\lambda>0$ and $\beta\geq 0$, we define
$$
\eta^{(\beta)}_\lambda(t,x,v):=t^{-\frac{\beta}{2}-2d} \exp\left(-\lambda\frac{|x-tv|^2+t^2|v|^2}{t^3}\right).
$$
By \eqref{density}, it is easy to see that there are constants $c_0,\lambda_0, c_1,\lambda_1>0$ such that (see \cite{RZ25})
\begin{align}\label{density-bound}
 c_0 \eta^{(0)}_{\lambda_0}(t,x,v)\leq   p_t(x,v)
    \le c_1 \eta^{(0)}_{\lambda_1}(t,x,v).
\end{align}
For a function $f:\mR_+\times\mR^{2d}\to\mR$,
we extend it by zero to negative times and define
\begin{align*}
    \cK^{(\beta)}_{\lambda} (f;\delta)
    :=\sup_{t\ge 0}
    \int_0^\delta \sup_{(x,v)\in\mR^{2d}}\int_{\mR^{2d}}
    \eta^{(\beta)}_\lambda(r,y,w)
    |f(t\pm r,x\pm y,v\pm w)|
    \dif y\dif w\dif r,
\end{align*}
where for simplicity of notation,
$$
|f(t\pm r,x\pm y,v\pm w)|:=\sum_{i,j,k=0,1}| f (t+(-1)^i
r, x +(-1)^j y,v+(-1)^k w)|.
$$

\bd\label{def:Kato}{\rm (Kato's class)}
For any $\beta\ge0$, we define
\begin{align*}
    \mK^d_\beta:=\left\{f\in L^{1}_{\rm loc}(\mR^{2d+1}):\cK^{(\beta)}_{\lambda} (f;1)<\infty,\, \lim_{\delta\downarrow0}\cK^{(\beta)}_{\lambda} (f;\delta)=0, \, \forall \lambda>0\right\}.
\end{align*}
\ed

\begin{examples}\label{ex:Lp}
For exponents $p_x,p_v\in[1,\infty]$, we use the mixed norm
\begin{align*}
    \|f\|_{L^{p_v}_vL^{p_x}_x}
    :=
    \left(
        \int_{\mR^d}
        \left(
            \int_{\mR^d}|f(x,v)|^{p_x}\dif x
        \right)^{p_v/p_x}
        \dif v
    \right)^{1/p_v},
\end{align*}
with the usual modifications when $p_x=\infty$ or $p_v=\infty$. 
By H\"older's inequality, it is easy to see that for some $C=C({d,p_x,p_v,q,\beta,\lambda})>0$
(see \cite[Lemma 2.6]{RZ25}),
\begin{align*}
    \cK^{(\beta)}_\lambda(f;\delta)
    \lesssim_C
    \delta^{\frac{2-\beta-\kappa}{2}}
    \|f\|_{L^q(\mR_+;L^{p_v}_vL^{p_x}_x)},
\end{align*}
where $\kappa:=\frac{2}{q}+\frac{3d}{p_x}+\frac{d}{p_v}<2-\beta$.
Hence,
\begin{align*}
    L^q(\mR_+;L^{p_v}_vL^{p_x}_x)\subset \mK_\beta .
\end{align*}
\end{examples}
\br\label{rmk:Kato01}
We note that $|f|^2\in \mK_0$ does not necessarily imply $f\in \mK_1$. For example, near the origin one may take
$$
f(t)=t^{-\frac12}|\log t|^{-1},\qquad 0<t<\tfrac12.
$$
Then $|f|^2$ is locally integrable near zero, whereas
$t^{-\frac12}f(t)\notin L^1_{\rm loc}([0,\infty))$. Thus $|f|^2\in \mK_0$ but $f\notin \mK_1$.
\er
\bl\label{lem:Kato_con}
Let $p\ge1$, let $\mu=(\mu_s)_{s\ge0}$ be a finite signed measure flow, and assume that $|f|^p\in\mK_\beta$ for some $\beta\ge0$.
Then
\begin{align*}
    \cK^{(\beta)}_{\lambda} (|f*\mu|^p;\delta)\le \cK^{(\beta)}_{\lambda} (|f|^p;\delta)\sup_{t\in\mR_+}\|\mu_t\|_{\var}^p.
\end{align*}
Here, consistently with the definition of the Kato norm, both $f$ and the
measure flow are extended by zero to negative times.
\el
\begin{proof}
    By H\"older's inequality, we have for any $z=(x,v)\in\mR^{2d}$,
\begin{align*}
    |(f*\mu)(r,z)|^p\le  (|f|^p*|\mu|)(r,z)\|\mu_r\|_{\var}^{p-1},
\end{align*}
    which by Fubini's theorem implies that
\begin{align*}
&\sup_{z\in\mR^{2d}}\int_{\mR^{2d}}
    \eta^{(\beta)}_\lambda(r,z')
    |(f*\mu)(t\pm r,z\pm z')|^p\dif z'\\
    &\qquad\le\|\mu_{t\pm r}\|_{\var}^{p-1} \sup_{z\in\mR^{2d}}\int_{\mR^{2d}}
    \eta^{(\beta)}_\lambda(r,z')(|f|^p*|\mu|)(t\pm r,z\pm z')\dif z'\\
    &\qquad\le \|\mu_{t\pm r}\|_{\var}^{p}\sup_{z\in\mR^{2d}}\int_{\mR^{2d}}
    \eta^{(\beta)}_\lambda(r,z')|f(t\pm r,z\pm z')|^p\dif z'.
\end{align*}
Thus by the definition, we obtain the desired estimate.
\end{proof}

Next, we show an important relation between Kato's classes and anisotropic Besov spaces.

\begin{lemma}\label{lem:kato-besov-embedding}
Let $\beta\in[0,2)$,  $\eps\in(0,2-\beta)$ and $f:\mR^{2d}\to\mR$ be a time independent function. 
\begin{enumerate}[(i)]
\item If $f\in\mK_\beta$, then $f\in \bC^{\beta-2}_a$. More precisely,
for every $\lambda>0$, there is a constant $C=C(d,\lambda,\beta)>0$ such that for every $\delta\in(0,1)$,
\begin{equation}\label{eq:kato-besov-convolution}
\|f\|_{\bC^{\beta-2}_a}
    \leq
    C\delta^{(\beta-2)/2}\cK^{(\beta)}_\lambda(f;\delta).
\end{equation}
\item If $0\leq f\in L^1_{\rm loc}(\mR^{2d})\cap\bC^{\beta-2+\eps}_a$, then $f\in\mK_\beta$. More precisely,
for every $\lambda>0$, there is a constant $C=C(d,\lambda,\beta,\eps)>0$ such that for every $\delta\in(0,1)$,
\begin{equation}\label{eq:kato-besov-convolution-2}
    \cK^{(\beta)}_\lambda(f;\delta)\leq {C} \delta^{\eps/2}
    \|f\|_{\bC^{\beta-2+\eps}_a}.
\end{equation}
\end{enumerate}
\end{lemma}

\begin{proof}
(i) Fix $t\in(0,1)$ and $\lambda>0$.
Note that for $r\in[t/2,t]$ and $(y,w)\in B_{t^{1/2}}^a$,
\begin{equation*}
    \frac{|y-rw|^2+r^2|w|^2}{r^3}
    \le \frac{(|y|+r|w|)^2+r^2|w|^2}{t^3/2^3}\leq 5\cdot 2^3.
\end{equation*}
Consequently,
\begin{equation*}
    \eta^{(\beta)}_\lambda(r,y,w)
    \ge
    c_{\lambda}
    t^{-\beta/2-2d},
    \qquad
    r\in[t/2,t],\quad (y,w)\in B_{t^{1/2}}^a .
\end{equation*}
Thus, by definition,
\begin{align*}
    \cK^{(\beta)}_\lambda(f;t)
    &\ge
    \int_{t/2}^{t}
        \sup_{z\in\mR^{2d}}
        \int_{\mR^{2d}}
            \eta^{(\beta)}_\lambda(r,u)
            |f(z-u)|\,\dif u\,\dif r
\\
    &\ge
    c_{\lambda}
    t^{1-\beta/2-2d}
    \sup_{z\in\mR^{2d}}
        \int_{B_{t^{1/2}}^a}
            |f(z-u)|\,\dif u .
\end{align*}
Hence
\begin{equation}\label{local_mass_est}
    \sup_{z\in\mR^{2d}}
        \int_{B_{t^{1/2}}^a}
            |f(z-u)|\,\dif u
    \le
    C_{\lambda}
    t^{\beta/2+2d-1}
    \cK^{(\beta)}_\lambda(f;t).
\end{equation}
We now estimate the Littlewood--Paley blocks. Fix $\delta\in(0,1)$.
For $j\ge-1$, we set
\begin{equation*}
 r_j:=2^{-j}\wedge \delta^{1/2},\quad   \Gamma_j:=
    \left\{
        \gamma_{k,\ell}
        :=
        \big(r_j^{3}k,r_j\ell\big)
        :
        k,\ell\in\mathbb Z^d
    \right\}.
\end{equation*}
For $\gamma=\gamma_{k,\ell}\in\Gamma_j$, define the half-open anisotropic
rectangle
\begin{align*}
    Q_j(\gamma)
    &:=
    \gamma
    +
    \left[-\tfrac{r_j^{3}}{2},\tfrac{r_j^3}{2}\right)^d
    \times
    \left[-\tfrac{r_j}{2},\tfrac{r_j}{2}\right)^d.
\end{align*}
By the definition of $B_{r_j}^a$, it is easy to see that for some $c_d>1$,
\begin{equation}\label{RR2}
    Q_j(\gamma)-\gamma=\left[-\tfrac{r_j^{3}}{2},\tfrac{r_j^3}{2}\right)^d
    \times
    \left[-\tfrac{r_j}{2},\tfrac{r_j}{2}\right)^d
    \subset
    B_{c_d r_j}^a.
\end{equation}
By \eqref{RR1} and using the partition
$\mR^{2d}=\bigcup_{\gamma\in\Gamma_j}Q_j(\gamma)$, we have
\begin{align*}
    \|\mathcal R_j^a f\|_\infty
    &\le
    \sup_{z\in\mR^{2d}}
    \sum_{\gamma\in\Gamma_j}
        \int_{Q_j(\gamma)}
            |(\cF^{-1}\phi_j^a)(u)|
            |f(z-u)|\,\dif u
\\
    &\le
    \sup_{z\in\mR^{2d}}
    \sum_{\gamma\in\Gamma_j}
        \sup_{u\in Q_j(\gamma)}|(\cF^{-1}\phi_j^a)(u)|
        \int_{Q_j(\gamma)}
            |f(z-u)|\,\dif u
\\
    &\le
    \sum_{\gamma\in\Gamma_j}
        \sup_{u\in Q_j(\gamma)}|(\cF^{-1}\phi_j^a)(u)|
    \sup_{z\in\mR^{2d}}
        \int_{B^a_{c_dr_j}}
            |f(z-h)|\,\dif h,
\end{align*}
where the last step is due to the change of variable and \eqref{RR2}.
Since $\phi_j^a(\xi)=\phi^a_0(2^{-aj}\xi)$ and $\phi_0^a\in C^\infty_c(\mR^{2d})$, for every $m>0$ and $u\in Q_j(\gamma)$,
\begin{equation}\label{0607:00}
    |(\cF^{-1}\phi_j^a)(u)|
    \lesssim_{
    C_{m,d}}2^{4dj}
    (1+2^j|u|_a)^{-m}\lesssim_{
    C_{m,d}}2^{4dj}
    (1+2^j|\gamma|_a)^{-m},
    \qquad
    j\ge-1.
\end{equation}
where we have used
\begin{equation*}
    1+2^j|u|_a
    \asymp
    1+2^j|\gamma|_a,
\end{equation*}
with constants independent of $j$ and $\gamma$. 
Moreover, we note that for
$\gamma_{k,\ell}=(r_j^3k,r_j\ell)$,
\begin{align*}
    2^j|\gamma_{k,\ell}|_a
    &=
    2^j
    \left(
        |r_j^3k|^{1/3}
        +
        |r_j\ell|
    \right)=(2^j r_j)(
    |k|^{1/3}+|\ell|),
\end{align*} which implies that for any $m>4d$
\begin{align}\label{0607:01}
\sum_{\gamma\in\Gamma_j}
        (1+2^j|\gamma|_a)^{-m}
    &=
    \sum_{k,\ell\in\mathbb Z^d}
        (1+(2^j r_j)(|k|^{1/3}+|\ell|))^{-m}\no\\
        &\le 2 \int_{\mR^{2d}}(1+(2^j r_j)(|x|^{1/3}+|y|))^{-m}\dif x\dif y\no\\
        &\le 2(2^j r_j)^{-4d} \int_{\mR^{2d}}(1+|x|^{1/3}+|y|)^{-m}\dif x\dif y\lesssim  (2^j r_j)^{-4d}.
\end{align}
This together with \eqref{0607:00} yields
$$
\sum_{\gamma\in\Gamma_j}
        \sup_{u\in Q_j(\gamma)}|(\cF^{-1}\phi_j^a)(u)|
\lesssim 2^{4dj}(2^j r_j)^{-4d}=r_j^{-4d}.
$$
By \eqref{local_mass_est} with
$t=(c_dr_j)^2$ and \eqref{0607:01} we obtain
\begin{align*}
    \|\mathcal R_j^a f\|_\infty
    &\lesssim
    r_j^{-4d} 
    r_j^{\beta+4d-2}
    \cK^{(\beta)}_\lambda(f;(c_dr_j)^2)\lesssim  r_j^{\beta-2}
    \cK^{(\beta)}_\lambda(f;c_d^2\delta).
\end{align*}
Since 
$
    r_j^{-1}\le 2^{j}+\delta^{-\frac12}\lesssim \delta^{-\frac12}2^{j},
$
we have
\begin{equation*}
    \|\mathcal R_j^a f\|_\infty
    \lesssim 
    2^{(2-\beta)j}\delta^{\frac{\beta-2}2}
    \cK^{(\beta)}_\lambda(f;c_d^2\delta).
    \qquad j\ge-1.
\end{equation*}
Taking the supremum over all dyadic blocks yields \eqref{eq:kato-besov-convolution}. 

(ii) Since $\beta-2+\eps<0$,
by Lemma \ref{lem:kinetic-besov-semigroup}, we have for every $T>0$,
\begin{equation}\label{0607:03}
    \sup_{0<t\le T}
    \Big(t^{-(\beta-2+\eps)/2}\|P_t f\|_\infty\Big)\lesssim\|f\|_{\bC^{\beta-2+\eps}_a}.
\end{equation}
For every
$\lambda>0$ and every choice of spatial signs
$\sigma_1,\sigma_2\in\{-1,1\}$, by \eqref{density-bound}, there exist constants
$C_\lambda,c_\lambda>0$ such that, for all $0<r\le T$,
\begin{equation*}
    \eta^{(\beta)}_\lambda(r,\sigma_1 y,\sigma_2 w)
    \le
    C_\lambda r^{-\beta/2}
    p_{c_\lambda r}(y,w).
\end{equation*}
Using this Gaussian comparison, by \eqref{0607:03}, for each spatial sign choice, we obtain
\begin{align*}
    \sup_{z\in\mR^{2d}}
    \int_{\mR^{2d}}
        \eta^{(\beta)}_\lambda(r,u)
        f(z\pm u)\,\dif u
    &\lesssim r^{-\beta/2}
    \|P_{c_\lambda r}f\|_\infty\lesssim
    r^{-1+\eps/2}
    \|f\|_{\bC^{\beta-2+\eps}_a}.
\end{align*}
Integrating the last estimate over $r\in(0,\delta)$ gives
\begin{align*}
    \cK^{(\beta)}_\lambda(f;\delta)
    &\lesssim
    \int_0^\delta
        r^{-1+\eps/2}\,\dif r\,
        \|f\|_{\bC^{\beta-2+\eps}_a}\lesssim
    \delta^{\eps/2}
    \|f\|_{\bC^{\beta-2+\eps}_a}.
\end{align*}
This proves the estimate in (ii) of Lemma~\ref{lem:kato-besov-embedding} and completes the proof. 
\end{proof}

Let $K\in\mK_0$.
For finite signed measures $\mu,\eta$ for which the following products are well defined, we set
\begin{align}\label{Def1}
\sA_\mu\eta:=(K*\mu)\eta+(K*\eta)\mu.
\end{align}
In the applications below $\mu$ is a probability density with positive Besov
regularity and $\eta$ is either a finite signed measure or a negative Besov
distribution, so the expression is interpreted through the estimates that
follow.
We have the following estimate.
\bl\label{Le10}
Let $K\in\mK_0$, $\mu\in\cP(\mR^{2d})$, and let
$\eta\in\bB^{\alpha,2}_{2;a}$ be a finite signed measure.
For any $\beta\in(1,2)$ and $\alpha\in(1-\beta-2d,-2d)$, there is a constant $C=C(d,\beta,\alpha)>0$ such that
\begin{align*}
\|\sA_\mu\eta\|_{\bB^{\alpha+\beta-2,2}_{2;a}}
\lesssim_C
\Big(\|K\|_{\bB^{\beta-2,\infty}_{\infty;a}}
      +\|K\|_{\bB^{\beta-2,\infty}_{2;a}}\Big)
\Big(\|\mu\|_{\bB^{2d+1,\infty}_{1;a}}
      +\|\mu\|_{\bB^{2d+1,\infty}_{2;a}}\Big)\|\eta\|_{\bB^{\alpha,2}_{2;a}},
\end{align*}
whenever the right-hand side is finite.
\el
\begin{proof}
By Lemma \ref{Le23}, we have
\begin{align}\label{Lm1}
\|(K*\mu)\eta\|_{\bB^{\alpha,2}_{2;a}}
&\lesssim
\|K*\mu\|_{\bB^{2d+\beta-1,\infty}_{\infty;a}}
\|\eta\|_{\bB^{\alpha,2}_{2;a}}\lesssim
\|K\|_{\bB^{\beta-2,\infty}_{\infty;a}}
\|\mu\|_{\bB^{2d+1,\infty}_{1;a}}
\|\eta\|_{\bB^{\alpha,2}_{2;a}},
\end{align}
and
\begin{align}\label{Lm2}
\|(K*\eta)\mu\|_{\bB^{\alpha+\beta-2,2}_{2;a}}\lesssim \|K*\eta\|_{\bB^{\alpha+\beta-2,2}_{\infty;a}}
\|\mu\|_{\bB^{2d+1,\infty}_{2;a}}
\lesssim \|K\|_{\bB^{\beta-2,\infty}_{2;a}}\|\eta\|_{\bB^{\alpha,2}_{2;a}}\|\mu\|_{\bB^{2d+1,\infty}_{2;a}}.
\end{align}
Since $\alpha+\beta-2<\alpha$, the Besov embedding
$\bB^{\alpha,2}_{2;a}\hookrightarrow \bB^{\alpha+\beta-2,2}_{2;a}$ applies to
the first term. Combining this with the two estimates above gives the desired
bound.
\end{proof}

\subsection{Kinetic SDEs with Kato-drifts}
In this section we recall some results for the following kinetic SDE with Kato drifts:
\begin{align}\label{SDE0}
\dif X_t=V_t \dif t,\quad\
\dif V_t=b(t,X_t,V_t)\dif t+\sqrt{2}\,\dif W_t, 
\end{align}
where $b=b_0+b_1:\mR_+\times\mR^{2d}\to\mR^d$ satisfies the conditions in Assumption \ref{ass:main}.

Under Assumption \ref{ass:main}, for any $\mu_0\in \cP(\mR^{2d})$, by \cite[Theorem 4.2]{RZ25},
there exists a unique weak solution to SDE \eqref{SDE0} in the corresponding weak sense (equivalently, Definition~\ref{def:SDE} with $K=0$).
Moreover, by the classical Girsanov theorem and the Markov property, we have the following a priori Krylov estimate.

\bl\label{Le27}
Let $Z=(X,V)$ be a weak solution of SDE \eqref{SDE0} in the corresponding weak sense.
Then under Assumption \ref{ass:main}, for every $T>0$, there are constants
\begin{align*}
 \lambda_0=\lambda_0(d,\kappa_0)\in(0,1),\qquad    C_{T,b}=C(d,\kappa_0,T,\cK^{(0)}_{\lambda_0}(|b_1|^2;\cdot))>0,
\end{align*}
such that for any non-negative $f\in \mK_0$ and any $0\le s\le t\le T$,
\begin{align}\label{Kry10}
    \mE\left[
        \int_s^t f(r,Z_r)\dif r
        \bigm| \sF_s
    \right]
    \le
    C_{T,b} \cK^{(0)}_{\lambda_0}\big(f;t-s\big).
\end{align}
\el

\begin{proof}
We first consider the reference kinetic SDE with the regular drift $b_0$:
\begin{align*}
\begin{cases}
    \dif \bar X_r=\bar V_r \dif r,\\
    \dif \bar V_r=b_0(r,\bar Z_r)\dif r+\sqrt{2}\,\dif W_r,
\end{cases}
\qquad
\bar Z_r=(\bar X_r,\bar V_r).
\end{align*}
For $0\le s\le t\le T$ and $\bar Z_s=z$, denote by $\bar\mE_{s,z}$ the expectation with respect to the law of this reference process starting from $z$ at time $s$. Since $b_0$ is Lipschitz with linear growth controlled by $\kappa_0$, the Krylov estimate for the reference equation gives that (see \cite[(4.2)]{RZ25} with $b_1=0$ for instance), for any non-negative $g\in\mK_0$,
$$
    \bar\mE_{s,z}\left[
        \int_s^t g(r,\bar Z_r)\dif r
    \right]
    \le
    C_0 \cK^{(0)}_{\lambda_0}\big(g;t-s\big),
$$
where $C_0=C_0(d,\kappa_0)$ and $\lambda_0=\lambda_0(d,\kappa_0)\in(0,1)$. 
Using the standard Khasminskii argument (see \cite{XZ20}),  one obtains that for any $m\in\mN$,
\begin{align}\label{Kry-b0-second}
    \bar\mE_{s,z}\left[
        \left(
            \int_s^t g(r,\bar Z_r)\dif r
        \right)^m
    \right]
    \le
    m!\left(C_0 \cK^{(0)}_{\lambda_0}\big(g;t-s\big)\right)^m.
\end{align}
Since $|b_1|^2\in\mK_0$, there is a $\delta_0>0$ small enough such that 
$$
    C_0\cK^{(0)}_{\lambda_0}(|b_1|^2;\delta_0)\le 1/12 .
$$
Let $0\leq s\leq t\leq T$
with $t-s\leq \delta_0$.
Thus by \eqref{Kry-b0-second}, we have
\begin{align*}
    \bar\mE_{s,z}\left[
        \left(
            \int_s^t 6|b_1(r,\bar Z_r)|^2\dif r
        \right)^m
    \right]
    \le
    m!\left(6C_0 \cK^{(0)}_{\lambda_0}\big(|b_1|^2;t-s\big)\right)^m\le m!2^{-m},
\end{align*}
which gives by Taylor's expansion
\begin{align*}
    \sup_z
    \bar\mE_{s,z}
    \exp\left\{
        6\int_s^t |b_1(r,\bar Z_r)|^2\dif r
    \right\}
    \le 2 .
\end{align*}
For general $0\leq s\leq t\leq T$ with $t-s>\delta_0$,
dividing any interval $[s,t]\subset[0,T]$ into subintervals of length at most $\delta_0$ 
and using the Markov property, we obtain
\begin{align}\label{global-exp-b1-simple}
    \sup_{0\le s\le t\le T}\sup_z
    \bar\mE_{s,z}
    \exp\left\{
        6\int_s^t |b_1(r,\bar Z_r)|^2\dif r
    \right\}
    \le 2^{\lceil\frac{T}{\delta_0}\rceil},
\end{align}
where $\lceil x\rceil$ is the smallest integer greater than $x$. 

Fix $s\in[0,t]$. Define the exponential martingale
\begin{align*}
    M_{s,t}
    :=
    \exp\left(
        \frac1{\sqrt{2}}\int_s^t b_1(r,\bar Z_r)\cdot\dif W_r
        -
        \frac14\int_s^t |b_1(r,\bar Z_r)|^2\dif r
    \right).
\end{align*}
By Girsanov's theorem, under the probability measure $M_{s,T}\dif \bar\mP_{s,z}$ on $\sF_T$, the process
\begin{align*}
    W_t^{b}:=W_t-\frac1{\sqrt{2}}\int_s^t b_1(r,\bar Z_r)\dif r,
    \qquad t\in[s,T],
\end{align*}
is a Brownian motion, and $\bar Z$ solves
\begin{align*}
    \dif \bar X_t=\bar V_t\dif t,\qquad
    \dif \bar V_t=\big(b_0(t,\bar Z_t)+b_1(t,\bar Z_t)\big)\dif t+\sqrt{2}\,\dif W_t^{b}.
\end{align*}
By weak uniqueness of \eqref{SDE0}, this law coincides with the law of $Z$ starting from $z$ at time $s$.
By Cauchy's inequality, the supermartingale property of stochastic exponentials, and \eqref{global-exp-b1-simple},
\begin{align}\label{M2-simple-bound}
    \sup_{0\le s\le t\le T}\sup_z
    \bar\mE_{s,z}M_{s,t}^2
    \le C_{T,b}<\infty .
\end{align}

Let $\mE_{s,z}$ denote the expectation with respect to the solution of \eqref{SDE0} starting from $z$ at time $s$. By Girsanov's theorem, Cauchy's inequality, \eqref{Kry-b0-second}, and \eqref{M2-simple-bound}, we get
\begin{align*}
    \mE_{s,z}\left[
        \int_s^t f(r,Z_r)\dif r
    \right]
    &=
    \bar\mE_{s,z}\left[
        M_{s,t}
        \int_s^t f(r,\bar Z_r)\dif r
    \right]                                                   \\
    &\le
    \left\{
        \bar\mE_{s,z}M_{s,t}^2
    \right\}^{1/2}
    \left\{
        \bar\mE_{s,z}
        \left[
            \left(
                \int_s^t f(r,\bar Z_r)\dif r
            \right)^2
        \right]
    \right\}^{1/2}                                             \\
    &\le
    C_{T,b}\,\cK^{(0)}_{\lambda_0}\big(f;t-s\big).
\end{align*}
Finally, by the Markovian property, we have
\begin{align*}
    \mE\left[
        \int_s^t f(r,Z_r)\dif r
        \bigm| \sF_s
    \right]
    =
    \mE_{s,Z_s}\left[
        \int_s^t f(r,Z_r)\dif r
    \right].
\end{align*}
Combining the previous estimate with this identity gives \eqref{Kry10}.
\end{proof}

As a corollary, we have the following important exponential integrability estimate for singular functionals
(see \cite[Lemma 3.5]{XZ20} and \cite[Corollary 3.5]{ZZ18} for a standard proof).

\bc(Khasminskii's estimate)\label{Kha1}
Under Assumption \ref{ass:main}, let $Z=(X,V)$ be the unique solution of SDE \eqref{SDE0}. 
For any $a>0$, $\eps\in(0,1)$ and any non-negative $f\in \mK_0$, there is a constant
$\delta_0=\delta_0(a,\eps,T,b,f)>0$ such that for all $0\leq s\leq t\leq T$ with $t-s\leq\delta_0$,
\begin{align}\label{Kha1-est}
    \mE\left[
        \exp\left\{
            a\int_s^t f(r,Z_r)\dif r
        \right\}
        \biggm|\cF_s
    \right]
    \le\frac{1}{1-\eps}.
\end{align}
\ec

\begin{proof}
By \eqref{Kry10} and
using the standard Khasminskii argument, we have for any $m\in\mN$,
$$
\mE\left[
        \left(
            \int_s^t f(r,Z_r)\dif r
        \right)^m\Big|\cF_s
    \right]
    \le
    m!\left(C_{T,b} \cK^{(0)}_{\lambda_0}\big(f;t-s\big)\right)^m.
$$
Since $\lim_{\delta\to 0}\cK^{(0)}_{\lambda_0}\big(f;\delta\big)=0$, for any $\eps\in(0,1)$ and $a>0$,
one can choose $\delta_0$ small enough so that 
$$
\cK^{(0)}_{\lambda_0}\big(f;\delta_0\big)\leq\eps/(aC_{T,b}).
$$
Thus, for all $0\leq s\leq t\leq T$ with $t-s\leq\delta_0$,
$$
\mE\left[
    \exp\left\{
        a\int_s^t f(r,Z_r)\dif r
    \right\}
    \biggm|\cF_s
\right]
    \leq\sum_{m=0}^\infty\eps^m=\left[
        1-
        \eps
    \right]^{-1}.
    $$
The proof is complete.
\end{proof}

\subsection{Relative entropy}
In this subsection we recall the definition and basic properties of relative entropy.
Let $E$ be a Polish space and $\mu,\nu$ be two probability measures on $E$. The relative entropy $\cH(\mu|\nu)$ is defined by
\begin{align}\label{Rela}
\cH(\mu|\nu):=
\left\{
\begin{aligned}
&\int_E \frac{\dif\mu}{\dif\nu}\log\Big(\frac{\dif\mu}{\dif\nu}\Big)\dif\nu,&\mu\ll\nu,\\
&\infty,&\mbox{otherwise.}
\end{aligned}
\right.
\end{align}
Since $x\mapsto x\log x$ is convex on $[0,\infty)$, by Jensen's inequality, we have $\cH(\mu|\nu)\geq 0$.

The following theorem collects the main properties of $\cH(\cdot|\cdot)$ used below (see \cite[Theorem 2.1(ii)]{BV05}, \cite[Lemma 1.4.3(a)]{DE97}, \cite[Lemma 3.9]{DM01} and \cite[Theorem 2.6]{BD19}).
\bt
\begin{enumerate}[(i)]
\item (Pinsker's inequality) For any two probability measures $\mu,\nu$,
\begin{align}\label{Pin0}
\|\mu-\nu\|^2_{\rm var}\leq 2\cH(\mu|\nu).
\end{align}
\item (Data processing inequality) For any  $\mu,\nu\in\cP(E)$ and Borel measurable function $f:E\to F$, 
\begin{align}\label{DPin}
\cH(\mu\circ f^{-1}|\nu\circ f^{-1})\le \cH(\mu|\nu). 
\end{align}
\item (Variational representation of the relative entropy) For any $\mu,\nu\in\cP(E)$,
\begin{align}\label{Var}
\cH(\mu|\nu)=\sup_{\psi\in\cB_b(E)}\left(\int_E\psi\dif\mu-\log\int_E\e^\psi\dif\nu\right),
\end{align}
where $\cB_b(E)$ is the set of all bounded Borel measurable functions.
\item (Additive inequality) Let $\mu^N$ be a symmetric probability measure on $E^N$ and $\mu\in\cP(E)$. Then for any $k\leq N$, 
\begin{align}\label{BB4}
\cH(\mu^{N,k}|\mu^{\otimes k})\leq\frac{2k}{N}\cH(\mu^N|\mu^{\otimes N}),
\end{align}
where $\mu^{N,k}$ is the marginal distribution of the first $k$-component of $\mu^N$.
\item (Chain rule) 
Let $E,F$ be two Polish spaces, $\mu,\nu\in \cP(E)$ and $P(x,\dif y), Q(x,\dif y)$ be two probability kernels on $E\times \sB(F)$. Then 
\begin{align}\label{ent:chain}
    \cH\left(P(x,\dif y)\mu(\dif x)|Q(x,\dif y)\nu(\dif x)\right)=\cH(\mu|\nu)+\int_{E} \cH\left(P(x,\cdot)|Q(x,\cdot)\right)\mu(\dif x).
\end{align}
\end{enumerate}
\et
We need the following corollary.
\begin{corollary}\label{Cor217}
Let $(X_r)_{r\geq 0}$ and $(Y_r)_{r\geq 0}$ be two continuous Markov processes on a Polish space $E$.
Let $0\leq s\leq t$. Assume that $\cH\big(\cL(X_{[0,t]})\,\big|\,\cL(Y_{[0,t]})\big)<\infty .$
Then
\begin{align}\label{ent:increment}
&\cH\big(\cL(X_{[0,t]})\,\big|\,\cL(Y_{[0,t]})\big)
-\cH\big(\cL(X_{[0,s]})\,\big|\,\cL(Y_{[0,s]})\big) \notag\\
&\qquad =
\cH\big(\cL(X_{[s,t]})\,\big|\,\cL(Y_{[s,t]})\big)
-\cH\big(\cL(X_s)\,\big|\,\cL(Y_s)\big).
\end{align}
\end{corollary}

\begin{proof}
For $x\in E$, denote by
\[
P^X_{s,t}(x,\cdot):=\cL\big((X_r)_{r\in[s,t]}\mid X_s=x\big),
\qquad
P^Y_{s,t}(x,\cdot):=\cL\big((Y_r)_{r\in[s,t]}\mid Y_s=x\big)
\]
the transition kernels on $C([s,t];E)$. By the Markov property,
\[
\cL(X_{[0,t]})
=
\cL(X_{[0,s]})(\dif\omega)\,
P^X_{s,t}(\omega_s,\dif\eta),\ \ 
\cL(Y_{[0,t]})
=
\cL(Y_{[0,s]})(\dif\omega)\,
P^Y_{s,t}(\omega_s,\dif\eta).
\]
Applying the chain rule for relative entropy yields
\begin{align*}
\cH\big(\cL(X_{[0,t]})\,\big|\,\cL(Y_{[0,t]})\big)  =
\cH\big(\cL(X_{[0,s]})\,\big|\,\cL(Y_{[0,s]})\big)
+
\int_E
\cH\big(P^X_{s,t}(x,\cdot)\,\big|\,P^Y_{s,t}(x,\cdot)\big)
\,\cL(X_s)(\dif x).
\end{align*}
On the other hand,
\[
\cL(X_{[s,t]})
=
\cL(X_s)(\dif x)\,P^X_{s,t}(x,\dif\eta),
\qquad
\cL(Y_{[s,t]})
=
\cL(Y_s)(\dif x)\,P^Y_{s,t}(x,\dif\eta).
\]
Applying the chain rule once more gives
\begin{align*}
\cH\big(\cL(X_{[s,t]})\,\big|\,\cL(Y_{[s,t]})\big)
=
\cH\big(\cL(X_s)\,\big|\,\cL(Y_s)\big) +
\int_E
\cH\big(P^X_{s,t}(x,\cdot)\,\big|\,P^Y_{s,t}(x,\cdot)\big)
\,\cL(X_s)(\dif x).
\end{align*}
Subtracting the two identities proves \eqref{ent:increment}.
\end{proof}

We recall the following entropy formula for the martingale solutions of classical SDEs,
which is a consequence of Girsanov's theorem (see \cite[Lemma 4.4 and Remark 4.5]{Lacker21} for the most general form, see also \cite[Lemma 4.2]{HRZ24}).
\bl\label{Lem34}
For $i=1,2$, let $b^i:\mR_+\times\mR^{2d}\to\mR^d$ be two Borel measurable functions.
Suppose that for any initial data $(x,v)\in\mR^{2d}$, there exists a unique weak solution to SDE \eqref{SDE0} with $b=b^2$. 
Let $\mu^1_0,\mu^2_0\in\cP(\mR^{2d})$ be two initial distributions.
For $i=1,2$, denote by SDE$_i$ the SDE \eqref{SDE0} with drift $b=b^i$.
Then, for any weak solutions $Z^i=(X^i, V^i)$ to SDE$_i$ ($i=1,2$) with $\cL(Z^i_0)=\mu^i_0$, and for any $t\geq 0$, letting $P^i_{[0,t]}:=\cL(Z^i_{[0,t]})$, we have
$$
\cH(P^1_{[0,t]}|P^2_{[0,t]})\leq\cH(\mu^1_0|\mu^2_0)+\frac14\mE\left(\int^t_0\bigl|b^1(s,Z_s^1)-b^2(s,Z_s^1)\bigr|^2\dif s\right).
$$
\el

\subsection{An abstract Hilbert-space subgaussian estimate}\label{sec:subgaussian}

This section is independent of the kinetic SDE setting and can be read on its own. We prove three abstract probabilistic estimates that form the backbone of our propagation of chaos and central limit arguments. The core idea is the following: if a family of conditionally centered Hilbert-space-valued random variables has uniform square-exponential moments, then empirical means of these variables satisfy uniform exponential moment bounds with constants that do not grow with the number of summands. This substitutes for classical concentration inequalities (which require boundedness or subgaussian tails of the individual variables) and is made possible by the anisotropic Kato--Krylov bounds of Section~\ref{sec:pre} which provide precisely the required conditional square-exponential integrability.

The following lemma is a conditional Hilbert-space version of the
standard implication from square-exponential integrability to
sub-Gaussian moment-generating-function bounds. This implication is well known;
see, for instance, \cite{Pinelis1994,Fukuda1990,Antonini1997} and the recent exposition \cite{MollenhauerSchillings2023}. We include the proof to keep track of the constants.

\begin{lemma}\label{lem:linform-improved}
Let $\mH$ be a separable Hilbert space, let $\cG$ be a sub-$\sigma$-field, and let $\xi$ be an $\mH$-valued random variable such that for some deterministic constants $a>0$ and $\Lambda\in(1,\infty)$,
\begin{equation}\label{eq:abstract-assumption}
\E[\xi\mid \cG]=0,
\qquad
\E\bigl[\e^{a\norm{\xi}_\mH^2}\mid\cG\bigr]\le \Lambda.
\end{equation}
Then for every $\mH$-valued $\cG$-measurable $\eta$,
\begin{equation}\label{eq:linform-bound}
\E\bigl[\e^{\inner{\eta}{\xi}_\mH}\mid\cG\bigr]
\le \exp\bigl(L\norm{\eta}_\mH^2\bigr)
\qquad\text{a.s.},
\end{equation}
where $L = \frac{2\log \Lambda + 1/2}{a}$,
and for every
$0\le \theta<a$,
\begin{equation}\label{eq:onestep-holder}
\E\bigl[
\exp(\theta\norm{\eta+\xi}_{\mH}^{2})
\mid \cG
\bigr]
\le
\exp\left\{
\frac{\log \Lambda}{a}\theta
+
\left(
\theta+
\frac{(8\log \Lambda+2)\theta^{2}}{a-\theta}
\right)
\norm{\eta}_{\mH}^{2}
\right\}.
\end{equation}
\end{lemma}

\begin{proof}
(i) Assume $\eta\not=0$. Let $X = \inner{\eta}{\xi}_\mH/\|\eta\|_\mH$. We aim to bound $\E[\e^{t X} \mid \cG]$ for any $t \in \mathbb{R}$. Note that $\E[X \mid \cG] = 0$ and $\E[\e^{a X^2} \mid \cG] \le \E[\e^{a\norm{\xi}_\mH^2} \mid \cG] \le \Lambda$. We divide the proof into two cases based on the magnitude of $t$.

\textbf{Case 1:} Suppose $t^2 \le a$. We employ a symmetrization argument. Let $X'$ be an independent copy of $X$ conditionally on $\cG$. By Jensen's inequality and $\E[X' \mid \cG] = 0$, we have
$$
1=\e^{-t\E[X' \mid \cG]}\leq \E[\e^{-tX'} \mid \cG].
$$
Hence,
\begin{equation*}
\E[\e^{tX} \mid \cG] \le \E[\e^{t(X-X')} \mid \cG] = \E[(\e^{t(X-X')}+\e^{-t(X-X')}) \mid \cG]/2.
\end{equation*}
Using the elementary inequality $(\e^{z}+\e^{-z})/2 \le \e^{z^2/2}$ and the algebraic inequality $(X-X')^2 \le 2X^2 + 2(X')^2$, we obtain
\begin{equation*}
\E[\e^{tX} \mid \cG] \le \E\left[\e^{\frac{t^2}{2}(X-X')^2} \Bigm| \cG\right] \le \E\left[\e^{t^2 X^2 + t^2 (X')^2} \Bigm| \cG\right] = \left( \E[\e^{t^2 X^2} \mid \cG] \right)^2.
\end{equation*}
Since $t^2 \le a$, Jensen's inequality implies
\begin{equation*}
\E[\e^{t^2 X^2} \mid \cG] = \E\left[ (\e^{a X^2})^{t^2/a} \Bigm| \cG\right] \le \left( \E[\e^{a X^2} \mid \cG] \right)^{t^2/a} \le \Lambda^{t^2/a}.
\end{equation*}
Substituting this back yields
\begin{equation*}
\E[\e^{tX} \mid \cG] \le (\Lambda^{t^2/a})^2 = \exp\left( \frac{2\log \Lambda}{a} t^2 \right).
\end{equation*}

\textbf{Case 2:} Suppose $t^2 > a$. By Young's inequality, $|tX| \le \frac{t^2}{2a} + \frac{a X^2}{2}$. Thus,
\begin{equation*}
\E[\e^{tX} \mid \cG] \le \E[\e^{|t X|} \mid \cG] \le \e^{t^2/(2a)} \E[\e^{a X^2 / 2} \mid \cG].
\end{equation*}
Again by Jensen's inequality, $\E[\e^{a X^2 / 2} \mid \cG] \le (\E[\e^{a X^2} \mid \cG])^{1/2} \le \Lambda^{1/2} = \e^{\frac{\log \Lambda}{2}}$. Therefore,
\begin{equation*}
\E[\e^{tX} \mid \cG] \le \exp\left( \frac{t^2}{2a} + \frac{\log \Lambda}{2} \right).
\end{equation*}
Because we are in the regime $t^2 > a$, we have $\frac{1}{2}\log \Lambda < \frac{t^2}{2a}\log \Lambda \le \frac{2\log \Lambda}{a} t^2$ (given $\Lambda > 1$). Hence,
\begin{equation*}
\frac{t^2}{2a} + \frac{\log \Lambda}{2}\le \frac{t^2}{2a} + \frac{2\log \Lambda}{a}t^2 = \left( \frac{2\log \Lambda + 1/2}{a} \right) t^2 = L t^2.
\end{equation*}
Combining both cases, we conclude that for all $t \in \mathbb{R}$, $\E[\e^{tX} \mid \cG] \le \exp(L t^2)$, establishing \eqref{eq:linform-bound}.

(ii) The case $\theta=0$ is trivial. Let $0<\theta<a$ and set
\[
p=\frac{a}{a-\theta},
\qquad
q=\frac{a}{\theta},
\qquad
\frac1p+\frac1q=1.
\]
Expanding the square gives
\[
\E\bigl[
\exp(\theta\norm{\eta+\xi}_{\mH}^{2})
\mid \cG
\bigr]
=
\exp(\theta\norm{\eta}_{\mH}^{2})
\E\bigl[
\exp(2\theta\inner{\eta}{\xi}_{\mH}
+\theta\norm{\xi}_{\mH}^{2})
\mid \cG
\bigr].
\]
By H\"older's inequality, we have
\begin{align*}
&\E\bigl[
\exp(2\theta\inner{\eta}{\xi}_{\mH}
+\theta\norm{\xi}_{\mH}^{2})
\mid \cG
\bigr]
\le
\left(
\E\bigl[
\exp(2p\theta\inner{\eta}{\xi}_{\mH})
\mid\cG
\bigr]
\right)^{1/p}
\left(
\E\bigl[
\exp(q\theta\norm{\xi}_{\mH}^{2})
\mid\cG
\bigr]
\right)^{1/q}.
\end{align*}
Since $q\theta=a$, the second factor is bounded by
\[
\left(
\E\bigl[
\exp(a\norm{\xi}_{\mH}^{2})
\mid\cG
\bigr]
\right)^{1/q}
\le
\Lambda^{\theta/a}
=
\exp\left(\frac{\log \Lambda}{a}\theta\right).
\]
For the first factor, the linear functional estimate gives
\[
\left(
\E\bigl[
\exp(2p\theta\inner{\eta}{\xi}_{\mH})
\mid\cG
\bigr]
\right)^{1/p}
\le
\exp\left(
4Lp\theta^{2}\norm{\eta}_{\mH}^{2}
\right).
\]
Combining these estimates proves \eqref{eq:onestep-holder}.
\end{proof}

The following lemma is a conditional Hilbert-space version of standard
exponential integrability estimates for sums of Banach-space valued
random variables; see, e.g., Kuelbs \cite{Kuelbs1978},
de Acosta \cite{deAcosta1980}, Ledoux--Talagrand \cite{LedouxTalagrand91},
and Pinelis \cite{Pinelis1994} for related results.
We include the proof in order to keep track of the explicit constants
under the square-exponential moment assumption.
For sub-Gaussian random variables in Hilbert spaces and concentration
of their squared norms under trace-class variance proxies, see also
Giuliano Antonini \cite{Antonini1997} and
Mollenhauer--Schillings \cite{MollenhauerSchillings2023}.

\begin{lemma}
\label{lem:abstract-average-holder}
Let $\xi_1,\dots,\xi_m$ be $\mH$-valued random variables.
Assume that, conditionally on $\cG$, they are independent and satisfy
for some deterministic constants $a>0$ and $\Lambda\in(1,\infty)$,
\[
\E[\xi_i\mid \cG]=0,
\qquad
\E\bigl[\exp(a\norm{\xi_i}_{\mH}^{2})\mid \cG\bigr]\le \Lambda,
\qquad i=1,\dots,m.
\]
Then, for every $D>\beta:=8\log \Lambda+2$ and every $m\ge1$,
\begin{equation}\label{eq:abstract-average-holder}
\E\left[
\exp\left(
\frac{a}{Dm}
\left\|
\sum_{i=1}^{m}\xi_i
\right\|_{\mH}^{2}
\right)
\Bigm|\cG
\right]
\le
\exp\left[
\frac{\log \Lambda}{\beta}
\log\frac{D}{D-\beta}
\right]
\qquad\text{a.s.}
\end{equation}
\end{lemma}

\begin{proof}
Let
\[
S_k:=\sum_{i=1}^{k}\xi_i,
\qquad
S_0:=0,
\qquad
\mathcal F_k:=\sigma(\cG,\xi_1,\dots,\xi_k).
\]
By conditional independence, $\xi_k$ satisfies the same conditional
centering and exponential-square estimate with respect to $\mathcal F_{k-1}$.
Applying Lemma~\ref{lem:linform-improved} with
$\eta=S_{k-1}$ and $\xi=\xi_k$, we obtain, for every $0\le\theta_k<a$,
\begin{align}
&\E\bigl[
\exp(\theta_k\norm{S_k}_{\mH}^{2})
\mid \mathcal F_{k-1}
\bigr]\le
\exp\left\{
A\theta_k
+
\left(
\theta_k+
\frac{\beta\theta_k^{2}}{a-\theta_k}
\right)
\norm{S_{k-1}}_{\mH}^{2}
\right\},
\label{eq:holder-step-Sk}
\end{align}
where
\[
A:=\frac{\log \Lambda}{a},
\qquad
\beta:=8\log \Lambda+2.
\]
Define the sequence $\theta_k$ backwards by
\begin{equation}\label{eq:holder-theta-recursion}
\theta_{k-1}
=
\theta_k+
\frac{\beta\theta_k^{2}}{a-\theta_k},
\qquad k=m,m-1,\dots,1,
\end{equation}
with terminal value
\begin{equation}\label{eq:holder-theta-terminal}
\theta_m:=\frac{a}{Dm},
\qquad D>\beta.
\end{equation}
We first check that all $\theta_k<a$. From \eqref{eq:holder-theta-recursion},
\[
\theta_{k-1}
=
\frac{\theta_k\bigl(a+(\beta-1)\theta_k\bigr)}
{a-\theta_k}.
\]
Hence
\begin{align}
\frac1{\theta_{k-1}}
&=
\frac{a-\theta_k}
{\theta_k\bigl(a+(\beta-1)\theta_k\bigr)} =
\frac1{\theta_k}
-
\frac{\beta}{a+(\beta-1)\theta_k}
\ge
\frac1{\theta_k}
-
\frac{\beta}{a}.
\label{eq:theta-reciprocal}
\end{align}
Iterating \eqref{eq:theta-reciprocal} gives
\[
\frac1{\theta_k}
\ge
\frac1{\theta_m}
-
\frac{\beta}{a}(m-k)
=
\frac{Dm}{a}
-
\frac{\beta}{a}(m-k)
=
\frac{(D-\beta)m+\beta k}{a}.
\]
Therefore
\begin{equation}\label{eq:theta-bound-holder}
\theta_k
\le
\frac{a}{(D-\beta)m+\beta k}
<a,
\qquad k=1,\dots,m.
\end{equation}
This validates the use of \eqref{eq:holder-step-Sk} at each step.

Using \eqref{eq:holder-step-Sk} and the definition of $\theta_{k-1}$,
we get
\[
\E\bigl[
\exp(\theta_k\norm{S_k}_{\mH}^{2})
\mid\cG
\bigr]
\le
\exp(A\theta_k)
\E\bigl[
\exp(\theta_{k-1}\norm{S_{k-1}}_{\mH}^{2})
\mid\cG
\bigr].
\]
Iterating from $k=m$ down to $k=1$ yields
\begin{equation}\label{eq:holder-iteration}
\E\bigl[
\exp(\theta_m\norm{S_m}_{\mH}^{2})
\mid\cG
\bigr]
\le
\exp\left(
A\sum_{j=1}^{m}\theta_j
\right).
\end{equation}
Note that by \eqref{eq:theta-bound-holder},
\begin{align}
\sum_{j=1}^{m}\theta_j
&\le
a\sum_{j=1}^{m}
\frac1{(D-\beta)m+\beta j} \le
a\int_{0}^{m}
\frac{dx}{(D-\beta)m+\beta x} =
\frac{a}{\beta}
\log\frac{D}{D-\beta}.
\label{eq:theta-sum-holder}
\end{align}
Combining \eqref{eq:holder-iteration} and \eqref{eq:theta-sum-holder}, and using
$A=\log \Lambda/a$, gives
\[
\E\bigl[
\exp(\theta_m\norm{S_m}_{\mH}^{2})
\mid\cG
\bigr]
\le
\exp\left[
\frac{\log \Lambda}{\beta}
\log\frac{D}{D-\beta}
\right].
\]
Finally, since $\theta_m=a/(Dm)$ and $S_m=\sum_{i=1}^{m}\xi_i$,
we obtain
\[
\E\left[
\exp\left(
\frac{a}{Dm}
\left\|
\sum_{i=1}^{m}\xi_i
\right\|_{\mH}^{2}
\right)
\Bigm|\cG
\right]
\le
\exp\left[
\frac{\log \Lambda}{\beta}
\log\frac{D}{D-\beta}
\right],
\]
which proves \eqref{eq:abstract-average-holder}.
\end{proof}

As a result, we also have the following estimate  about the U-statistics see \cite{deLaPenaMontgomerySmith95}.
\begin{lemma}\label{lem:clt-centered-double-sum}
Let $\mH$ be a separable Hilbert space, let $E$ be a measurable space, and
let $\xi$ be an $E$-valued random variable. For each $m\ge2$, let
$\xi_1,\dots,\xi_m$ be i.i.d. random variables with the same law as
$\xi$. Let $h:E\times E\to \mH$ be a measurable kernel satisfying the two
centering conditions
$$
    \mE h(z,\xi)=0,\qquad
    \mE h(\xi,z')=0
$$
for every frozen pair $z,z'\in E$. Assume that, for some $a>0$ and
$K<\infty$,
$$
    \sup_z\mE\exp\bigl(a\norm{h(z,\xi)}_\mH^2\bigr)
    +
    \sup_{z'}\mE\exp\bigl(a\norm{h(\xi,z')}_\mH^2\bigr)
    \le K .
$$
Then there is a constant $c_0=c_0(a,K)>0$,
independent of $m$, such that, for every $m\ge2$,
\begin{align}\label{eq:clt-centered-double-sum}
    \mE\exp\left(
        \frac{c_0}m
        \norm{
        \sum_{i\ne j}h(\xi_i,\xi_j)
        }_{\mH}
    \right)
    \le 2 .
\end{align}
\end{lemma}

\begin{proof}
Set
$$
    S_m:=\sum_{i\ne j}h(\xi_i,\xi_j).
$$
\noindent\emph{Step 1: Reduction to a decoupled randomized form.}
Let $(\widetilde\xi_j)_{j=1}^m$ be an independent copy of
$(\xi_j)_{j=1}^m$, and let
$(\varepsilon_i)_{i\ge1}$, $(\varepsilon'_j)_{j\ge1}$ be independent
Rademacher sequences, independent of all $\widetilde\xi, \xi$'s. Set
$$
    \widetilde S_m
    :=
    \sum_{i\ne j}
    \varepsilon_i\varepsilon'_j h(\xi_i,\widetilde\xi_j).
$$
Let $r\geq 2$ and $L^r(\mH):=L^r(\Omega;\mH) $. 
We claim that for a universal constant $C$ independent of $r$,
$$
    \norm{S_m}_{L^r(\mH)}
    \le C\norm{\widetilde S_m}_{L^r(\mH)}.
$$
Indeed, define
$$
    D_m:=\sum_{i\ne j}h(\xi_i,\widetilde\xi_j).
$$
By the decoupling inequality of de la Pe{\~n}a and Montgomery-Smith
\cite[Theorem~1]{deLaPenaMontgomerySmith95}, applied with order $2$ and
Banach space $\mH$, there is a universal constant $C_2$ independent of $r,m$ such that
$$
    \norm{S_m}_{L^r(\mH)}\le C_2\norm{D_m}_{L^r(\mH)},
$$
We next symmetrize $D_m$ twice. Conditionally on $(\xi_i)_{i=1}^m$, the
variables
$$
    Y_j:=\sum_{i\ne j}h(\xi_i,\widetilde\xi_j),
    \qquad j=1,\ldots,m,
$$
are independent and centered in $\mH$, since $\mE h(z,\xi)=0$ for every
frozen $z$. Applying the standard Banach-valued symmetrization inequality
\cite[Chapter~6, Lemma~6.3]{LedouxTalagrand91} in this conditional probability
space, and then integrating over $(\xi_i)_{i=1}^m$, yields
$$
    \norm{D_m}_{L^r(\mH)}
    \le
    2\left\|
    \sum_{j=1}^m\varepsilon'_j
    \sum_{i\ne j}h(\xi_i,\widetilde\xi_j)
    \right\|_{L^r(\mH)} .
$$
Conditionally on $(\widetilde\xi_j,\varepsilon'_j)_{j=1}^m$, the variables
$$
    X_i:=\sum_{j\ne i}\varepsilon'_j h(\xi_i,\widetilde\xi_j),
    \qquad i=1,\ldots,m,
$$
are independent and centered, now by the condition $\mE h(\xi,z')=0$ for
every frozen $z'$. Applying the same inequality in this conditional
probability space and then integrating gives
$$
    \norm{D_m}_{L^r(\mH)}
    \le 4\norm{\widetilde S_m}_{L^r(\mH)}.
$$
Combining the decoupling and symmetrization estimates proves the claim.

\smallskip
\noindent\emph{Step 2: Estimate of the decoupled randomized form.}
Let $\emptyset\not=J\subset\{1,\ldots,m\}$. For fixed $z\in E$, set
$$
    Z_J(z):=\sum_{j\in J}\varepsilon'_j h(z,\widetilde\xi_j).
$$
The summands are independent, centered, $\mH$-valued random variables and
satisfy
$$
    \mE\exp\bigl(a\norm{\varepsilon'_j h(z,\widetilde\xi_j)}_\mH^2\bigr)
    =
    \mE\exp\bigl(a\norm{h(z,\xi)}_\mH^2\bigr)
    \le K .
$$
Lemma \ref{lem:abstract-average-holder} therefore gives, uniformly in $z$, for some $c=c(a,K)>0$,
$$
    \mP(\norm{Z_J(z)}_\mH>u)
    \le\mE\exp\left(c\frac{\norm{Z_J(z)}_\mH^2}{|J|}\right)\exp\left(-c\frac{u^2}{|J|}\right)\le
    C\exp\left(-c\frac{u^2}{|J|}\right),\qquad u>0,
$$
Integrating this Gaussian tail
gives, for $r\ge2$,
$$
    \sup_{z\in E}\norm{Z_J(z)}_{L^r(\mH)}
    \le C(a,K)\sqrt{r\,|J|}.
$$
Set
$$
    B_i:=\sum_{j\ne i}\varepsilon'_j h(\xi_i,\widetilde\xi_j),
$$
Since $\xi_i$ is independent of
$(\widetilde\xi_j,\varepsilon'_j)_{j\ne i}$, the preceding estimate with
$J=\{1,\ldots,m\}\setminus\{i\}$ yields
$$
    \norm{B_i}_{L^r(\mH)}
    \le C(a,K)\sqrt{rm},
    \qquad i=1,\ldots,m .
$$
Let
$$
    \mathcal G:=
    \sigma\bigl(
    \xi_1,\ldots,\xi_m,\widetilde\xi_1,\ldots,\widetilde\xi_m,
    \varepsilon'_1,\ldots,\varepsilon'_m
    \bigr).
$$
The vectors $B_i$ are $\mathcal G$-measurable, and
$(\varepsilon_i)_{i=1}^m$ is independent of $\mathcal G$. Hence the
Hilbert-space Khintchine--Kahane inequality gives, conditionally on
$\mathcal G$,
$$
    \mE\left(\left\|
        \sum_{i=1}^m\varepsilon_i B_i
    \right\|_{\mH}^r\Big|\cG\right)
    \le
    C r^{r/2}
    \left(\sum_{i=1}^m\norm{B_i}_\mH^2\right)^{r/2}.
$$
Taking the $L^r$-norm in the remaining variables and using Minkowski's
inequality in $L^{r/2}$,
$$
\begin{aligned}
    \norm{\widetilde S_m}_{L^r(\mH)}
    &\le
    C\sqrt r
    \left\|
        \left(\sum_{i=1}^m\norm{B_i}_\mH^2\right)^{1/2}
    \right\|_{L^r} \le
    C\sqrt r
    \left(\sum_{i=1}^m\norm{B_i}_{L^r(\mH)}^2\right)^{1/2}  \\
    &\le C\sqrt r\,(m\cdot rm)^{1/2}
    \le C rm .
\end{aligned}
$$
Therefore, for some $C_0>0$,
$$
    \norm{S_m}_{L^r(\mH)}
    \le C_0rm,\qquad m,r\ge2 .
$$
Let $X_m:=m^{-1}\norm{S_m}_\mH$. Then for every integer $n\ge2$,
$$
    \mE X_m^n\le (C_0n)^n .
$$
Using Stirling's bound $n!\ge(n/\e)^n$, for $c_0=1/(2\e C_0)$,
$$
\begin{aligned}
    \mE \e^{c_0X_m}
    &\le
    1+c_0\mE X_m
    +\sum_{n=2}^\infty \frac{c_0^n}{n!}\mE X_m^n \le
    1+2c_0C_0
    +\sum_{n=2}^\infty \bigl(c_0\e C_0\bigr)^n \leq 2.
\end{aligned}
$$
Hence,
$$
    \sup_{m\ge2}\mE
    \exp\left(
        \frac{c_0}{m}\norm{S_m}_\mH
    \right)
    \le 2.
$$
This is exactly \eqref{eq:clt-centered-double-sum}.
\end{proof}

\section{Proofs of Theorem \ref{thm:well-posed} and Theorem \ref{thm:main}}\label{sec:well}

This section contains the proofs of our two main results. Both proofs rely on the same three ingredients introduced in the preliminaries: the relative entropy formula of Lemma~\ref{Lem34} (which bounds the path-space entropy between two kinetic SDEs by the squared $L^2$-difference of their drifts), the Krylov estimate of Lemma~\ref{Le27} (which controls functionals of the solution by their Kato norms), and the subgaussian estimate of Lemma~\ref{lem:abstract-average-holder} (which turns conditional square-exponential integrability into uniform exponential moments). The well-posedness proof in the next subsection uses only the first two ingredients in a freezing-and-Picard iteration. The propagation of chaos proof in Section~\ref{sec:PoC} additionally requires the Hilbert-space subgaussian estimate to control the empirical interaction fluctuation without boundedness assumptions on $K$.

First we fix some notation. Let $T>0$ and $d\in\mN$.
We write
$$
\cC^{d}_{T}:=C([0,T];\mR^{2d}).
$$
Let $\pi_t(\omega):=\omega(t)$ be the canonical process on $\cC^{d}_{T}$. 
For a probability measure $\bP\in\cP(\cC^d_T)$, we denote by $\bP_{[0,t]}$ the restriction of $\bP$ to $\cC^d_t$, and by $\mu^\bP_t$
the marginal law of $\bP$ at time $t$, i.e., $\mu^\bP_t:=\bP\circ \pi_t^{-1}.$

\subsection{Proof of Theorem \ref{thm:well-posed}}
The well-posedness is proved by a freezing-and-Picard iteration on a short time interval, followed by concatenation. The core estimate is the relative entropy bound of Lemma~\ref{Lem34}: for two frozen equations with drifts differing by $K*(\mu^1-\mu^2)$, the entropy between their path laws is controlled by the Kato norm of the drift difference. Combined with the Krylov estimate (Lemma~\ref{Le27}), this yields a contraction in total variation on a small time interval whose length depends only on the Kato characteristics of $b$ and $K$, not on the initial law. The Picard iterates therefore form a Cauchy sequence, whose limit is identified as a weak solution. Global existence follows by concatenating the local solutions. Uniqueness is then a direct consequence of the same entropy estimate applied to two candidate solutions.

For a measure flow $\nu=(\nu_t)_{t\ge 0}$, write
\begin{align*}
    B_\nu(t,z):= b(t,z)+(K*\nu_t)(z).
\end{align*}

\noindent{\bf Existence.}
Let $T_0>0$ be a small time, to be determined below.
Let $\bP^0\in\cP(\cC^d_{T_0})$ be such that
\begin{align*}
    \bP^0\big(\pi_t=\pi_0,\ \forall t\in[0,T_0]\big)=1,
    \quad
    \bP^0\circ \pi_0^{-1}=\mu_0.
\end{align*}
Let $\mu^0_t:=\bP^0\circ\pi_t^{-1}.$
Assume that $\bP^n$ and $\mu^n_t:=\bP^n\circ\pi_t^{-1}$ have been constructed. Define $\bP^{n+1}$ as the law of the unique weak solution to the frozen equation
\begin{align*}
\begin{cases}
    \dif X^{n+1}_t=V^{n+1}_t\,\dif t,\quad \cL(Z^{n+1}_0)=\mu_0, \\
    \dif V^{n+1}_t
    =
    B_{\mu^n}(t,Z^{n+1}_t)\,\dif t
    +\sqrt{2}\,\dif W_t.
\end{cases}
\end{align*}
Note that by Assumption \ref{ass:main} and Lemma \ref{lem:Kato_con},
$$
    \cK^{(1)}_\lambda(|K*\mu^n|;\delta)
    \le \cK^{(1)}_\lambda(|K|;\delta),\quad\cK^{(0)}_\lambda(|K*\mu^n|^2;\delta)
    \le \cK^{(0)}_\lambda(|K|^2;\delta)=:\bar \kappa_\lambda(\delta).
$$
Hence,
$$
    \cK^{(0)}_\lambda(|b_1+K*\mu^n|^2;\delta)
    \le 2\cK^{(0)}_\lambda(|b_1|^2;\delta)+2\bar \kappa_\lambda(\delta).
$$
Thus the Krylov, Khasminskii and Girsanov constants used below can be chosen independently of the Picard index $n$.

For $n,m\ge 0$,
by Lemma \ref{Lem34} and Lemma \ref{Le27}, we have for every $t\in[0,T_0]$,
\begin{align*}
    \cH\big(\bP^{n+m+1}_{[0,t]}\mid \bP^{n+1}_{[0,t]}\big)
	    &\leq \frac14\mE\left(\int^t_0\bigl|K*(\mu^{n+m}-\mu^n)(s,Z_s^{n+m+1})\bigr|^2\dif s\right)\\
    &\le
    C_{T_0} \cK^{(0)}_{\lambda_0}
    \Big(
        |K*(\mu^{n+m}-\mu^n)|^2;t
    \Big)\\
    & \le
    C_{T_0}\bar \kappa_{\lambda_0}(t)
    \sup_{0\le r\le t}
    \|\mu^{n+m}_r-\mu^n_r\|_{\var}^2,
\end{align*}
where $C_{T_0}$ is independent of $n,m$ and increasing in the time horizon.
Noting that
\begin{align*}
    \sup_{0\le r\le t}
    \|\mu^{n+m}_r-\mu^n_r\|_{\var}^2
    \le
    \|\bP^{n+m}_{[0,t]}-\bP^n_{[0,t]}\|_{\var}^2,
\end{align*}
by Pinsker's inequality \eqref{Pin0},
\begin{align*}
    \|\bP^{n+m+1}_{[0,t]}-\bP^{n+1}_{[0,t]}\|_{\var}^2
    &\le
    2\cH\big(\bP^{n+m+1}_{[0,t]}\mid \bP^{n+1}_{[0,t]}\big)\le
    C_{T_0}\bar \kappa_{\lambda_0}(t)
    \|\bP^{n+m}_{[0,t]}-\bP^n_{[0,t]}\|_{\var}^2 .
\end{align*}
Since $\lim_{\delta\to 0}\bar\kappa_{\lambda_0}(\delta)=0$, choosing $T_0$ smaller if necessary, we may assume
\begin{align*}
    q:=C_{T_0}\bar \kappa_{\lambda_0}(T_0)<1.
\end{align*}
Taking $t=T_0$ and iterating the above inequality gives
\begin{align*}
    \|\bP^{n+m}_{[0,T_0]}-\bP^n_{[0,T_0]}\|_{\var}
    \le
    q^{n/2}
    \|\bP^m_{[0,T_0]}-\bP^0_{[0,T_0]}\|_{\var}.
\end{align*}
Since the total variation distance between two probability measures is uniformly bounded, we obtain
\begin{align*}
    \lim_{n\to\infty}\sup_{m\ge 1}
    \|\bP^{n+m}_{[0,T_0]}-\bP^n_{[0,T_0]}\|_{\var}
    =0.
\end{align*}
Thus $(\bP^n)_{n\ge 0}$ is a Cauchy sequence in total variation on $\cP(\cC^d_{T_0})$. Hence there exists $\bP\in\cP(\cC^d_{T_0})$ such that
\begin{align*}
    \lim_{n\to\infty}
    \|\bP^n-\bP\|_{\var}=0.
\end{align*}
It remains to identify $\bP$ as the law of a solution to the McKean--Vlasov equation. Let $\bQ$ be the law of the unique weak solution to the frozen equation
$$    
\dif X_t=V_t\,\dif t,\quad \dif V_t=B_\mu(t,Z_t)\dif t+\sqrt{2}\,\dif W_t,
    $$
where $\mu_t:=\bP\circ\pi_t^{-1}$ and $\cL(Z_0)=\mu_0$.
By Lemma \ref{Lem34} and Lemma \ref{Le27}, we have for every $t\in[0,T_0]$,
\begin{align*}
    \cH\big(\bP^{n+1}_{[0,t]}\mid \bQ_{[0,t]}\big)
    &\le
    C_{T_0} \;{\cK^{(0)}_{\lambda_0}}
    \Big(
        |K*(\mu^n-\mu)|^2;t
    \Big) \\
    &\le
    q
    \sup_{0\le r\le t}
    \|\mu^n_r-\mu_r\|_{\var}^2\le
    q
    \|\bP^n-\bP\|_{\var}^2
    \longrightarrow 0 .
\end{align*}
Therefore, by Pinsker's inequality,
\begin{align*}
    \bP^{n+1}_{[0,t]}\longrightarrow \bQ_{[0,t]}
    \quad\text{in total variation}.
\end{align*}
On the other hand, $\bP^{n+1}_{[0,t]}\to \bP_{[0,t]}$ in total variation. Hence
\begin{align*}
    \bQ_{[0,t]}=\bP_{[0,t]},\qquad t\in[0,T_0].
\end{align*}
In particular, $\bP$ is the law of a weak solution to \eqref{kMVSDE} on $[0,T_0]$.

Finally, for a general finite time horizon $T>0$, we repeat the above local construction on consecutive intervals
\begin{align*}
    [0,T_0],\ [T_0,2T_0],\ \ldots,\ [kT_0,(k+1)T_0],
\end{align*}
using at each step the terminal law of the previous interval as the new initial law. Since the choice of $T_0$ depends only on the Kato characteristics of $b_1$ and $K$, on $d$, and on $\kappa_0$, but not on the initial law, the same local argument applies on every interval. By the standard disintegration and concatenation procedure for weak solutions, these local solutions can be glued into a weak solution on $[0,T]$. Since $T>0$ is arbitrary, existence holds globally in time.

\noindent{\bf Uniqueness.}
Let $Z^1,Z^2$ be two weak solutions to \eqref{kMVSDE} with the same initial law $\mu_0$. 
Applying the same entropy estimate as above to the two law flows $\mu^1$ and
$\mu^2$ gives, for $t\le T_0$,
\[
    \|\mu^1_t-\mu^2_t\|_{\var}^2
    \le
    C_{T_0}\bar\kappa_{\lambda_0}(t)
    \sup_{0\le r\le t}\|\mu^1_r-\mu^2_r\|_{\var}^2 .
\]
Choosing $T_0$ so that the prefactor is strictly smaller than one yields
$\mu^1_t=\mu^2_t$ on $[0,T_0]$. Repeating the argument on consecutive
intervals gives
\begin{align*}
    \mu^1_t=\mu^2_t,\qquad t\ge 0.
\end{align*}
Finally, since the two measure flows coincide, both $Z^1$ and $Z^2$ solve the same linear kinetic SDE with drift $B_\mu$. By the weak uniqueness of the linear kinetic SDE with Kato drift, their path laws coincide. Thus uniqueness in law holds for \eqref{kMVSDE}.

\subsection{Proof of Theorem \ref{thm:main}}\label{sec:PoC}

The proof proceeds in three steps. First, Lemma~\ref{lem:short-time-exp} establishes a uniform short-time exponential moment bound for the empirical fluctuation functional $F_{s,t}^N$ defined in \eqref{FstN-def}. This is the technical heart of the proof: using the conditional independence structure and the Hilbert-space subgaussian estimate of Lemma~\ref{lem:abstract-average-holder}, we show that the squared fluctuation of the empirical interaction field satisfies an exponential moment bound with constants independent of $N$. Second, this short-time estimate is combined with the entropy formula of Lemma~\ref{Lem34} to obtain the $k/N$ relative entropy bound on short intervals. Third, the Markov property and the entropy chain rule extend the estimate to arbitrary finite times. The proof of Lemma~\ref{lem:short-time-exp} itself treats the case $N=1$ by the Khasminskii estimate and $N\ge2$ by conditioning on the first particle and applying the Hilbert-space subgaussian machinery.

\bl\label{lem:short-time-exp}
For $0\leq s\leq t\leq T$, let $\bP_{[s,t]}$ be the law of $Z$ on $C([s,t];\mR^{2d})$. Fix $N\in\mN$ and 
for $z^1,\ldots,z^N\in C([s,t];\mR^{2d})$, define a functional
\begin{align}\label{FstN-def}
    F_{s,t}^N(z^1,\cdots,z^N)
    :=
    \frac{1}N
    \int_s^{t}
    \left|
        \sum_{i=1}^N
        \left[
            K(r,z_r^1-z_r^i)
            -
            (K(r,\cdot)*\mu_r)(z_r^1)
        \right]
    \right|^2
    \dif r,
\end{align}
where $\mu_r$ is the law of $Z_r$. 
Under Assumption \ref{ass:main}, 
there exist $\delta_0=\delta_0(T,d,\kappa_0,b,K)>0$ and
$C=C(T,d,\kappa_0,b,K)>0$ such that, for every
$0\leq s\leq t\leq T$ with $t-s\leq\delta_0$ and every $N\ge1$,
\begin{align}\label{eq:uniform-exp-short}
    \mE_{\bP_{[s,t]}^{\otimes N}}
    \exp\left\{
        F_{s,t}^N
    \right\}
    \le C.
\end{align}
Moreover, there is a small constant $c_0>0$ such that for each $N\geq 1$,
\begin{align}\label{Eq1}
    \mE_{\bP_{[0,T]}^{\otimes N}}
    \exp\left\{c_0
        F_{0,T}^N
    \right\}
    \le C.
\end{align}
\el

\begin{proof}
For $N=1$, it follows from Corollary \ref{Kha1}, Assumption \ref{ass:main}, and the convention $K(r,0)=0$.
Below we fix $0\leq s\leq t\leq T$ and $N\ge2$. Let $Z^1,\ldots,Z^N$ be $N$ independent copies of $Z$. 
Then
\begin{align}\label{SW1}
\mE_{\bP_{[s,t]}^{\otimes N}}
    \exp\left\{
        F_{s,t}^N
    \right\}=\mE \exp\left\{
        F_{s,t}^N(Z^1,\cdots, Z^N)
    \right\}
\end{align}
Set
\begin{align*}
    \mH_{s,t}:=L^2([s,t];\mR^d),
    \qquad
    \cG:=\sigma(Z_r^1,\ s\le r\le t).
\end{align*}
For $i=1,\ldots,N$, define the $\mH_{s,t}$-valued random variable
\begin{align}\label{xi-def-PoC}
    \xi_i(r)
    :=
    K(r,Z_r^1-Z_r^i)
    -
    (K(r,\cdot)*\mu_r)(Z_r^1),
    \qquad
    s\le r\le t .
\end{align}
Since $K(r,0)=0$,
\begin{align}\label{xi1-def}
    \xi_1(r)
    =
    -
    (K(r,\cdot)*\mu_r)(Z_r^1).
\end{align}
For $i=2,\ldots,N$, conditionally on $\cG$, the random variables
$\xi_i$, $i=2,\ldots,N$, are independent, and noting that
\begin{align*}
    \mE\left[
        K(r,Z_r^1-Z_r^i)
        \bigm| \cG
    \right]
    =
    \int_{\mR^{2d}}K(r,Z_r^1-z)\mu_r(\dif z)
    =
    (K(r,\cdot)*\mu_r)(Z_r^1),
\end{align*}
we have
\begin{align}\label{xi-centered}
    \mE[\xi_i(r)\mid\cG]=0,
    \qquad i=2,\ldots,N,
\end{align}
which implies, by \eqref{FstN-def} and \eqref{xi-def-PoC}, that
\begin{align}\label{F-split}
    F_{s,t}^N(Z^1,\ldots,Z^N)
    =
    \frac1N
    \left\|
        \sum_{i=1}^N\xi_i
    \right\|_{\mH_{s,t}}^2
    \le
    \frac1{N-1}
    \left\|
        \sum_{i=2}^N\xi_i
    \right\|_{\mH_{s,t}}^2
    +
    \|\xi_1\|_{\mH_{s,t}}^2,
\end{align}
and by \eqref{SW1} and H\"older's inequality,
\begin{align}\label{SW5}
    \mE_{\bP_{[s,t]}^{\otimes N}}
    \exp\left\{
        F_{s,t}^N
    \right\}
    &\le
    \left(
        \mE
        \exp\left\{
            \frac2{N-1}
            \left\|
                \sum_{i=2}^N\xi_i
            \right\|_{\mH_{s,t}}^2
        \right\}
    \right)^{1/2}
    \left(
        \mE
        \exp\left\{
            2\|\xi_1\|_{\mH_{s,t}}^2
        \right\}
    \right)^{1/2}.
\end{align}
We now verify the hypotheses of Lemma \ref{lem:abstract-average-holder}.
Note that for $i\ge2$,
\begin{align*}
    |\xi_i(r)|^2
    \le
    2|K(r,Z_r^1-Z_r^i)|^2+2|(K(r,\cdot)*\mu_r)(Z_r^1)|^2.
\end{align*}
By Jensen's inequality, we have
\begin{align}\label{m-Jensen}
    |(K(r,\cdot)*\mu_r)(Z_r^1)|^2
    \le
    \mE\left[
        |K(r,Z_r^1-Z_r^i)|^2
        \bigm| \cG
    \right].
\end{align}
Consequently, using Jensen's inequality once more, we get
\begin{align}\label{xi-cond-exp}
    \mE\left[
        \exp\left\{
            a\|\xi_i\|_{\mH_{s,t}}^2
        \right\}
        \bigm| \cG
    \right]
    \le
    \left[\mE
    \exp\left\{
        2a\int_s^{t}
        |K(r,Z^1_r-Z^i_r)|^2
        \dif r
    \right\}\bigm|\cG\right]^2,
\end{align}
Now, fix a path $\omega\in C([s,t];\mR^{2d})$ and consider the function on $\mR_+\times\mR^{2d}$
\begin{align*}
    g_\omega(r,z)
    :=|K(r,\omega_r-z)|^2.
\end{align*}
By translation invariance in the Kato norm,
\begin{align}\label{Kato-shift-K}
    \cK_\lambda^{(0)}
    \big(g_\omega;\delta\big)=\cK_\lambda^{(0)}(|K|^2;\delta).
\end{align}
Applying \eqref{Kha1-est} to $g_\omega$ with exponential coefficient $2a$,
we may choose $\delta_0=\delta_0(a,T,b,K)>0$ such that, whenever
$0\leq s\leq t\leq T$ and $t-s\leq\delta_0$,
$$
\mE
    \exp\left\{
        2a\int_s^{t}
        |K(r,\omega_r-Z_r)|^2
        \dif r
    \right\}
=\mE
    \exp\left\{
        2a\int_s^{t}
        g_\omega(r,Z_r)
        \dif r
    \right\}
    \le 2.
$$
Substituting this into \eqref{xi-cond-exp}
we obtain
$$
\mE\left[
        \exp\left\{
            a\|\xi_i\|_{\mH_{s,t}}^2
        \right\}
        \bigm| \cG
    \right]
    \le 4.
$$
Choose $a>2(8\log 4+3)$. 
Applying Lemma \ref{lem:abstract-average-holder} conditionally on $\cG$, with
\begin{align*}
    m=N-1,
    \qquad
    \Lambda=4,
    \qquad
    D=8\log 4+3>\beta=8\log 4+2,
\end{align*}
we obtain
\begin{align}\label{sum-xi-subgaussian}
    \mE\left[
        \exp\left\{
            \frac2{N-1}
            \left\|
                \sum_{i=2}^N\xi_i
            \right\|_{\mH_{s,t}}^2
        \right\}
        \bigm| \cG
    \right]
    \le \exp\left[
\frac{\log \Lambda}{\beta}
\log\frac{D}{D-\beta}
\right].
\end{align}
It remains to control $\xi_1$. By \eqref{xi1-def}, \eqref{m-Jensen}, and Jensen's inequality,
\begin{align}\label{SW4}
    \exp\left\{
        2\|\xi_1\|_{\mH_{s,t}}^2
    \right\}
    &\le
    \exp\left\{
        2\int_s^{t}
        \mE\left[
            |K(r,Z_r^1-Z^2_r)|^2
            \bigm|\cG
        \right]\dif r
    \right\}                                                    \no \\
    &\le
    \mE\left[
        \exp\left\{
            2\int_s^{t}
            |K(r,Z_r^1-Z_r^2)|^2
            \dif r
        \right\}
        \bigm| \cG
    \right]\leq 2.
\end{align}
Substituting \eqref{sum-xi-subgaussian} and \eqref{SW4} into \eqref{SW5}, we obtain the desired short-time estimate \eqref{eq:uniform-exp-short}.

It remains to derive the global bound \eqref{Eq1}. Choose a partition
$0=t_0<t_1<\cdots<t_m=T$ such that $t_\ell-t_{\ell-1}\le\delta_0$.
Since
\[
    F_{0,T}^N=\sum_{\ell=1}^{m}F_{t_{\ell-1},t_\ell}^N,
\]
H\"older's inequality and \eqref{eq:uniform-exp-short} give, for
$c_0\le 1/m$,
\[
    \mE_{\bP_{[0,T]}^{\otimes N}}
    \exp\{c_0F_{0,T}^N\}
    \le
    \prod_{\ell=1}^{m}
    \left(
        \mE_{\bP_{[t_{\ell-1},t_\ell]}^{\otimes N}}
        \exp\{mc_0F_{t_{\ell-1},t_\ell}^N\}
    \right)^{1/m}
    \le C .
\]
This proves \eqref{Eq1}.
\end{proof}

Now we are ready to give the proof.

\begin{proof}[Proof of Theorem \ref{thm:main}]
For the particle system \eqref{IPS}, we write
\begin{align*}
    \bP_{[s,t]}^{N,k}
    &:=
    \cL(Z^{N,1}_{[s,t]},\ldots,Z^{N,k}_{[s,t]}),
    \qquad k=1,\ldots,N,\\
    \mu_t^{N,k}
    &:=
    \cL(Z_t^{N,1},\ldots,Z_t^{N,k}),
    \qquad
    \mu_t=\cL(Z_t).
\end{align*}
Since the initial law is exchangeable and the particle system is weakly unique,
$\bP_{[s,t]}^{N,N}$ is exchangeable for every $0\le s\le t$. By
Lemma \ref{Lem34}, applied to the $N$-particle system and to $N$ independent nonlinear copies, we have
\begin{align}
    \cH\left(
        \bP_{[s,t]}^{N,N}
        \mid
        \bP_{[s,t]}^{\otimes N}
    \right)
    &\le
    \cH\left(
        \mu_s^{N,N}
        \mid
        \mu_s^{\otimes N}
    \right)
    +
    \frac14
    \sum_{i=1}^N
    \mE_{\bP_{[s,t]}^{N,N}}
    \int_s^t
    \left|
        (K(r,\cdot)*(\mu_r^N-\mu_r))(z_r^i)
    \right|^2
    \dif r\no\\
    &=\cH\left(
        \mu_s^{N,N}
        \mid
        \mu_s^{\otimes N}
    \right)
    +
    \frac N4
    \mE_{\bP_{[s,t]}^{N,N}}
    \int_s^t
    \left|
        (K(r,\cdot)*(\mu_r^N-\mu_r))(z_r^1)
    \right|^2
    \dif r,\label{eq:HN-first-bound}
\end{align}
where $\mu_r^N:=\frac1N\sum_{j=1}^N\delta_{z_r^j}$
is the empirical measure of the particle system at time $r$, since $\bP_{[s,t]}^{N,N}$ is exchangeable. 
By \eqref{FstN-def}, we have
\begin{align}\label{Mb1}
    F_{s,t}^N(z^1,\ldots,z^N)
    =
    N
    \int_s^t
    \left|
        (K(r,\cdot)*(\mu_r^N-\mu_r))(z_r^1)
    \right|^2
    \dif r .
\end{align}
By the variational representation of relative entropy \eqref{Var}, applied first to $F_{s,t}^N\wedge M$ and then letting $M\to\infty$,
\begin{align*}
    \mE_{\bP_{[s,t]}^{N,N}}F_{s,t}^N
    \le
    \cH\left(
        \bP_{[s,t]}^{N,N}
        \mid
        \bP_{[s,t]}^{\otimes N}
    \right)
    +
    \log
    \mE_{\bP_{[s,t]}^{\otimes N}}
    \exp(F_{s,t}^N).
\end{align*}
Combining this with \eqref{eq:HN-first-bound} and \eqref{Mb1}, we obtain
$$
    \cH\left(
        \bP_{[s,t]}^{N,N}
        \mid
        \bP_{[s,t]}^{\otimes N}
    \right)
    \le
    \cH\left(
        \mu_s^{N,N}
        \mid
        \mu_s^{\otimes N}
    \right)
    +
    \frac14\cH\left(
        \bP_{[s,t]}^{N,N}
        \mid
        \bP_{[s,t]}^{\otimes N}
    \right)
    +
    \frac14
    \log
    \mE_{\bP_{[s,t]}^{\otimes N}}
    \exp(F_{s,t}^N),
$$
and therefore
$$
    \cH\left(
        \bP_{[s,t]}^{N,N}
        \mid
        \bP_{[s,t]}^{\otimes N}
    \right)
    \le
    \frac43\cH\left(
        \mu_s^{N,N}
        \mid
        \mu_s^{\otimes N}
    \right)
    +
    \frac13
    \log
    \mE_{\bP_{[s,t]}^{\otimes N}}
    \exp(F_{s,t}^N).
$$
By Lemma \ref{lem:short-time-exp}, there is a $\delta_0>0$ small enough so that for $t-s\leq\delta_0$,
\begin{align}\label{eq:short-time-entropy}
\cH\left(
        \bP_{[s,t]}^{N,N}
        \mid
        \bP_{[s,t]}^{\otimes N}
    \right)
    \le
    \frac43\cH\left(
        \mu_s^{N,N}
        \mid
        \mu_s^{\otimes N}
    \right)
	    +\frac13 \log C .
\end{align}
In particular, 
$$
\cH\left(
        \bP_{[0,\delta_0]}^{N,N}
        \mid
        \bP_{[0,\delta_0]}^{\otimes N}
    \right)
    \le
    \frac43\cH\left(
        \mu_0^{N,N}
        \mid
        \mu_0^{\otimes N}
    \right)
	    +\frac13 \log C .
$$
By Corollary \ref{Cor217}, we can extend this to arbitrary finite times. Indeed, by \eqref{ent:increment},
\begin{align*}
\cH\left(
        \bP_{[0,t]}^{N,N}
        \mid
        \bP_{[0,t]}^{\otimes N}
    \right)
  -    \cH\left(
        \bP_{[0,s]}^{N,N}
        \mid
        \bP_{[0,s]}^{\otimes N}
    \right)
=
    \cH\left(
        \bP_{[s,t]}^{N,N}
        \mid
        \bP_{[s,t]}^{\otimes N}
    \right)-
    \cH\left(
        \mu_s^{N,N}
        \mid
        \mu_s^{\otimes N}
    \right)
\end{align*}
which, together with \eqref{eq:short-time-entropy} and the data processing inequality, yields
\begin{align*}
    \cH\left(
        \bP_{[0,2\delta_0]}^{N,N}
        \mid
        \bP_{[0,2\delta_0]}^{\otimes N}
    \right)
	  \le \frac43 \cH\left(
        \bP_{[0,\delta_0]}^{N,N}
        \mid
        \bP_{[0,\delta_0]}^{\otimes N}
	    \right)+\frac13 \log C.
\end{align*}
Let $m=\lceil t/\delta_0\rceil$ and divide $[0,t]$ into $m$ subintervals of
length at most $\delta_0$. Iterating the preceding estimate gives
\begin{align*}
    \cH\left(
        \bP_{[0,t]}^{N,N}
        \mid
        \bP_{[0,t]}^{\otimes N}
    \right)
    \le
    C_t\left(
        \cH(\mu_0^{N,N}\mid\mu_0^{\otimes N})+1
    \right),
\end{align*}
where $C_t$ is independent of $N$ and of the initial exchangeable law.
Finally, applying the additive entropy inequality \eqref{BB4} on the path
space $C([0,t];\mR^{2d})$ gives, for every $1\le k\le N$,
\begin{align*}
    \cH(\bP_{[0,t]}^{N,k}\mid \bP_{[0,t]}^{\otimes k})
    \le
    \frac{2k}{N}
    \cH(\bP_{[0,t]}^{N,N}\mid\bP_{[0,t]}^{\otimes N})
    \le
    \frac{C_t k}{N}
    \left(
        \cH(\mu_0^{N,N}\mid\mu_0^{\otimes N})+1
    \right),
\end{align*}
which is \eqref{PoC:main}.
\end{proof}

\section{Central Limit Theorem for Fluctuations}\label{sec:fluctuation-consequence}

In this section we prove the central limit theorem for the empirical fluctuation field $\eta_t^N=\sqrt N(\mu_t^N-\mu_t)$. Unlike the propagation of chaos estimate of Theorem~\ref{thm:main}, 
the CLT requires second-order control: we must show that the scaled empirical measure converges to a Gaussian process in suitable negative Besov spaces. The proof first establishes the regularity of the limiting law, then derives uniform estimates for $\eta_t^N$ and the nonlinear interaction term by transferring exponential moment bounds from the independent particle system to the interacting system. The final step combines tightness, identification of limit points through the linearized equation, the stable martingale CLT, and uniqueness of the limiting martingale problem.
After establishing the CLT, we further obtain quantitative Berry--Esseen bounds for finite-dimensional projections of the fluctuation field.

Throughout this section, we consider the following nonlinear SDE
\begin{align}\label{SDE10}
\dif X_t=V_t\dif t,\ \ \dif V_t=(K*\mu_t)(Z_t)\dif t+\sqrt{2}\dif W_t,\ \ Z_0\sim\mu_0\in\cP(\mR^{2d}),
\end{align}
where $K:\mR^{2d}\to\mR^d$ is a time-independent singular interaction kernel satisfying
\begin{assumption}\label{ass:clt}
For some $\beta\in(1,2)$, it holds that
\begin{align}\label{KK0}
K\in \mK_\beta\cap \bB^{\beta-2,\infty}_{2;a},\ \ |K|^2\in \mK_0,\ \ \mu_0\in\bB^{2d+1,\infty}_{1;a}\cap \bB^{2d+1,\infty}_{2;a}.
\end{align}
\end{assumption}
By It\^o's formula, the law $\mu_t$ of $Z_t$ solves the following nonlinear Fokker-Planck equation:
\begin{align}\label{eq:clt-law-fpe}
    \partial_t\mu_t
    =
    \Delta_v\mu_t
    -
    v\cdot\nabla_x\mu_t
    -
    \div_v((K*\mu_t)\mu_t).
\end{align}
We first establish the following regularity estimate about $\mu_t$.
\begin{theorem}
\label{lem:mu-besov-apriori}
Under Assumption \ref{ass:clt}, $\mu_t$ has a density, still
denoted by $\mu_t$, and for every $T>0$ and $\gamma>2d+1$, there is a constant $C>0$ such that
for all $t\in(0,T]$,
\begin{align}\label{eq:mu-besov-smoothing}
\|\mu_t\|_{\bB^{\gamma,\infty}_{1;a}}+\|\mu_t\|_{\bB^{\gamma,\infty}_{2;a}}\lesssim_C t^{-\frac{\gamma-2d-1}{2}}.
\end{align}
\end{theorem}
\begin{proof}
First, Lemma \ref{lem:kato-besov-embedding} gives
\begin{align}\label{Hg1}
K \in \mK_\beta \subset \bB^{\beta-2,\infty}_{\infty;a}.
\end{align}
Let $\beta_0 \in (1-\beta, 0)$. Since $\mu_0 \in \bB^{0,\infty}_{1;a}\subset\bB^{\beta_0,\infty}_{1;a}$, 
by \cite[Theorem 1.3 and Theorem 3.13]{HRZ}, for each $\gamma\geq 0$, 
there exists a constant $C>0$ such that for all $t\in(0,T]$,
\begin{align}\label{0622:00}
\|\mu_t\|_{\bB^{\gamma,1}_{1;a}}+\|\mu_t\|_{\bB^{\gamma,1}_{2;a}} \lesssim t^{-\frac{\gamma-\beta_0}{2}}.
\end{align}
We now apply \eqref{0622:00} to improve the singularity \eqref{eq:mu-besov-smoothing} 
around the time zero when the initial value is regular. 
Let $s\in[0,T)$ and $t \in [0, T-s]$. By Duhamel's formula, we have
$$
\mu_{t+s} = P^*_t \mu_s - \int_0^t P^*_{t-r} \div_v \big((K * \mu_{s+r}) \mu_{s+r}\big) \, \dif r.
$$
Let $0 \le \gamma_1<\gamma$.
By Lemma \ref{lem:kinetic-besov-semigroup}, we have for any $p \in [1,2]$,
\begin{align}\label{Hg2}
\|\mu_{t+s}\|_{\bB^{\gamma,1}_{p;a}} \lesssim \|P_t\mu_s\|_{\bB^{\gamma,1}_{p;a}} + \int_0^t (t-r)^{-\frac{1+\gamma-\gamma_1}{2}} \|(K * \mu_{s+r}) \mu_{s+r}\|_{\bB^{\gamma_1,\infty}_{p;a}} \, \dif r.
\end{align}
Let $p'=p/(p-1)\in[2,\infty]$. By \eqref{Hg1} and Assumption \ref{ass:clt}, we have
$$
\|K\|_{\bB^{\beta-2,\infty}_{p';a}}\leq\|K\|^{2/p'}_{\bB^{\beta-2,\infty}_{2;a}}\|K\|^{1-2/p'}_{\bB^{\beta-2,\infty}_{\infty;a}}<\infty.
$$
Thus, using \eqref{Besov-con} and \eqref{product}, we obtain
$$
\begin{aligned}
\|(K * \mu_{s+r}) \mu_{s+r}\|_{\bB^{\gamma_1,\infty}_{p;a}}
&\lesssim \|K * \mu_{s+r}\|_{\bB^{\gamma_1,\infty}_{\infty;a}} \|\mu_{s+r}\|_{L^p} + \|K * \mu_{s+r}\|_{L^\infty} \|\mu_{s+r}\|_{\bB^{\gamma_1,\infty}_{p;a}} \\
&\lesssim \|K\|_{\bB^{\beta-2,\infty}_{p';a}} \Big( \|\mu_{s+r}\|_{\bB^{2-\beta+\gamma_1,1}_{p;a}} \|\mu_{s+r}\|_{L^p}
+ \|\mu_{s+r}\|_{\bB^{2-\beta,1}_{p;a}} \|\mu_{s+r}\|_{\bB^{\gamma_1,\infty}_{p;a}} \Big)\\
&\lesssim \|\mu_{s+r}\|_{\bB^{2-\beta+\gamma_1,1}_{p;a}} \Big(\|\mu_{s+r}\|_{L^p}
+ \|\mu_{s+r}\|_{\bB^{2-\beta,1}_{p;a}}\wedge\|\mu_{s+r}\|_{\bB^{\gamma_1,\infty}_{p;a}} \Big),
\end{aligned}
$$
Substituting this into \eqref{Hg2} gives
\begin{align}\label{0622:02}
\begin{aligned}
\|\mu_{t+s}\|_{\bB^{\gamma,1}_{p;a}}
\lesssim \|P_t\mu_s\|_{\bB^{\gamma,1}_{p;a}} 
 &+ \int_0^t (t-r)^{-\frac{1+\gamma-\gamma_1}{2}} 
 \|\mu_{s+r}\|_{\bB^{2-\beta+\gamma_1,1}_{p;a}}\\
&\times \Big(\|\mu_{s+r}\|_{L^p}
+ \|\mu_{s+r}\|_{\bB^{2-\beta,1}_{p;a}}\wedge\|\mu_{s+r}\|_{\bB^{\gamma_1,\infty}_{p;a}} \Big)\dif r.
\end{aligned}
\end{align}
Choose $\beta_0\in((1-\beta)/2,0)$ and
$\gamma_0\in(0,\beta+2\beta_0-1)$. Letting $s=\gamma_1=0$ and
$\gamma=\gamma_0$ in \eqref{0622:02}, by Lemma
\ref{lem:kinetic-besov-semigroup} and \eqref{0622:00}, we get
\begin{align*}
\|\mu_t\|_{L^p}\lesssim\|\mu_t\|_{\bB^{\gamma_0,1}_{p;a}} 
&\lesssim \|\mu_0\|_{\bB^{\gamma_0,1}_{p;a}} + \int_0^t (t-r)^{-\frac{1+\gamma_0}{2}} 
\|\mu_r\|_{\bB^{2-\beta,1}_{p;a}} \|\mu_r\|_{L^p}  \dif r\\
&\lesssim \|\mu_0\|_{\bB^{\gamma_0,1}_{p;a}}
+\int_0^t(t-r)^{-\frac{1+\gamma_0}{2}}
r^{-\frac{2-\beta-\beta_0}{2}}r^{\frac{\beta_0}{2}}\,\dif r\lesssim 1.
\end{align*}
Next, choosing $s=\gamma_1=0$ and $\gamma = 2-\beta$  in \eqref{0622:02}  and using the above estimate give
$$
\|\mu_t\|_{\bB^{2-\beta,1}_{p;a}} 
\lesssim \|\mu_0\|_{\bB^{2-\beta,1}_{p;a}} 
+ \int_0^t (t-r)^{-\frac{3-\beta}{2}} \|\mu_r\|_{\bB^{2-\beta,1}_{p;a}} \, \dif r,
$$
which, by Gr\"onwall's inequality of Volterra type, implies
$$
\|\mu_t\|_{\bB^{2-\beta,1}_{p;a}}
\lesssim \|\mu_0\|_{\bB^{2,\infty}_{p;a}} \lesssim 1.
$$
Substituting this into \eqref{0622:02}, we obtain, for any $s>0$ and $\gamma>2-\beta$,
\begin{align}\label{0622:03}
\|\mu_{t+s}\|_{\bB^{\gamma,1}_{p;a}}
\lesssim \|P_t\mu_s\|_{\bB^{\gamma,1}_{p;a}} + \int_0^t (t-r)^{-\frac{3-\beta}{2}} \|\mu_{s+r}\|_{\bB^{\gamma,1}_{p;a}} \, \dif r.
\end{align}
By Gr\"onwall's inequality of Volterra type again, 
there is a constant $C_T>0$ such that for all $0<s\leq t\leq T$,
$$
\|\mu_{t+s}\|_{\bB^{\gamma,1}_{p;a}}\lesssim_{C_T} \|P_t\mu_s\|_{\bB^{\gamma,1}_{p;a}}
+\int^t_0(t-r)^{-\frac{3-\beta}{2}}\|P_r\mu_s\|_{\bB^{\gamma,1}_{p;a}}\dif r.
$$
From this and by Lemma \ref{lem:kinetic-besov-semigroup}, we derive the desired estimates.
\end{proof}

Now consider the particle system associated with SDE \eqref{SDE10}:
$$
 \dif X_t^{N,i}=V_t^{N,i}\dif t,\qquad
        \dif V_t^{N,i}=(K*\mu^N_t)(Z_t^{N,i})\dif t+\sqrt{2}\dif W_t^i,\ \ i=1,\cdots, N,
$$
where $Z^{N,i}_t:=(X^{N,i}_t,V^{N,i}_t)$ and 
$$
Z^N_0:=(Z^{N,1}_0,\cdots,Z^{N,N}_0)\sim \mu_0^{N,N}\in\cP((\mR^{2d})^N).
$$
Let $\mu_t^N:=\frac1N\sum_{i=1}^N\delta_{Z_t^{N,i}}$ be the empirical measure.
Consider the measure-valued process
$$
    \eta_t^N:=\sqrt N(\mu_t^N-\mu_t).
$$
The aim of this section is to show that $\eta_t^N$ weakly converges to an infinite-dimensional Ornstein--Uhlenbeck process $\eta_t$, where $\eta_t$ formally solves
the following linear SPDE:
\begin{align}\label{eq:linearized-clt-formal}
    \partial_t\eta_t
    =\Delta_v\eta_t
    -v\cdot\nabla_x\eta_t
    -\div_v(\sA_{\mu_t}\eta_t)
    -\sqrt{2}\,\div_v(\sqrt{\mu_t}\,\xi),
\end{align}
where $\sA_\mu\eta$ is defined by \eqref{Def1} and
$\xi$ is an $\mR^d$-valued white noise on $\mR_+\times\mR^{2d}$ with independent components. 

Indeed, let $\varphi\in C^1(\mR_+; C^\infty_c(\mR^{2d}))$.
Applying It\^o's formula to $\langle\eta_t^N,\varphi_t\rangle$ and subtracting
the deterministic equation for $\mu_t$ gives
\begin{align}\label{eq:clt-finite-equation}
\begin{split}
    \langle\eta_t^N,\varphi_t\rangle
    &=
    \langle\eta_0^N,\varphi_0\rangle
    +
    \int_0^t\langle\eta_s^N,(\p_s+\Delta_v+v\cdot\nabla_x)\varphi_s\rangle\,\dif s\\
    &
    +\sqrt N\int_0^t
    \left\langle
        \nabla_v\varphi_s,(K*\mu_s^N)\mu^N_s-(K*\mu_s)\mu_s\right\rangle\,\dif s
+    M_t^{N,\varphi},
 \end{split}
\end{align}
where
\begin{align}\label{eq:clt-finite-mart}
    M_t^{N,\varphi}
    :=
    \frac{\sqrt{2}}{\sqrt N}\sum_{i=1}^N
    \int_0^t\nabla_v\varphi_s(Z_s^{N,i})\cdot\dif W_s^i.
\end{align}
Intuitively, the martingale law of large numbers suggests that, as $N\to\infty$,
$$
M_t^{N,\varphi}\to \sqrt{2}\int^t_0\int_{\mR^{2d}}\sqrt{\mu_s(z)}\<\nabla_v\varphi_s(z)\cdot\xi(\dif s,\dif z)\>.
$$
Moreover, with the notation \eqref{Def1}, we formally also have
$$
\sqrt N\int_0^t
    \left\langle
        \nabla_v\varphi_s,(K*\mu_s^N)\mu^N_s-(K*\mu_s)\mu_s\right\rangle\,\dif s
\rightarrow \int_0^t
    \left\langle
        \nabla_v\varphi_s,\sA_{\mu_s}\eta_s\right\rangle\,\dif s.
$$

We now give the precise definition of a solution to SPDE \eqref{eq:linearized-clt-formal}.

\begin{definition}[Martingale solution]\label{def:clt-mart-sol}
Let $T>0$. A process $\eta$ is called a
martingale solution of SPDE \eqref{eq:linearized-clt-formal} on $[0,T]$ if
\begin{align}\label{Re1}
    \eta\in \bigcap_{\alpha<-2d}L^\infty([0,T]; L^2(\Omega;\bB^{\alpha,2}_{2;a})).
\end{align}
Moreover, for every $\varphi\in C^1([0,T]; \bB^{2d+2,2}_{2;a})$ and $t\in[0,T]$,
\begin{align}\label{eq:linearized-clt-weak}
\begin{split}
    \langle\eta_t,\varphi_t\rangle
    &=
    \langle\eta_0,\varphi_0\rangle
    +
    \int_0^t
        \langle\eta_s,(\p_s+\Delta_v+v\cdot\nabla_x+\sA^*_{\mu_s}\cdot\nabla_v)\varphi_s\rangle\,\dif s
    +M_t^\varphi,
\end{split}
\end{align}
where $M^\varphi$ is a centered continuous Gaussian martingale. The Gaussian martingale field is independent of $\eta_0$ and satisfies, for all admissible test functions $\varphi,\psi$,
\begin{align}\label{eq:clt-noise-covariance}
    \mE\big[M_t^\varphi M_s^\psi\big]
    =
    2\int_0^{t\wedge s}
        \langle\mu_r,\nabla_v\varphi_r\cdot\nabla_v\psi_r\rangle\,\dif r.
\end{align}
\end{definition}

\br\rm
By Assumption \ref{ass:clt}, Lemma \ref{Le10}, and
\eqref{eq:mu-besov-smoothing},
there is a $\delta\in(0,1)$ such that
$$
\|\sA_{\mu_s}\eta_s\|_{\bB^{\alpha+\beta-2,2}_{2;a}}\lesssim s^{-\delta},\ \ s\in(0,T].
$$
Thus, for any $\varphi\in C^1([0,T]; \bB^{2d+2,2}_{2;a})$,
$$
\int_0^t|\langle \eta_s,\sA^*_{\mu_s}\cdot\nabla_v\varphi_s\rangle|\,\dif s
=\int_0^t|\langle \sA_{\mu_s}\eta_s,\nabla_v\varphi_s\rangle|\,\dif s
<\infty,\ \ a.s.
$$  
\er

The aim of this section is to show the following central limit theorem.

\begin{theorem}[Central limit theorem]\label{thm:clt}
Let $\mu_0^{N,N}=\mu_0^{\otimes N}$ and $T>0$. Under Assumption \ref{ass:clt}, for every
$\ell<0$ and $\alpha<-2d$, the fluctuation fields
\[
    \eta^N_t=\sqrt N(\mu_t^N-\mu_t)
\]
converge in distribution in $C([0,T];\bB^{\alpha-3,2}_{2;a,\ell-1}(\mR^{2d}))$
to the unique martingale solution $\eta$ of
\eqref{eq:linearized-clt-formal} in the sense of Definition~\ref{def:clt-mart-sol},
where the initial centered Gaussian field $\eta_0$ has variance
\begin{align}\label{initial:cor}
       \mE\bigl|\langle\eta_0,\varphi\rangle|^2
    =
    \langle\mu_0,\varphi^2\rangle
    -|\langle\mu_0,\varphi\rangle|^2,\quad \varphi\in C^\infty_c(\mR^{2d}).
\end{align}
\end{theorem}
We shall use the following consequence of Khasminskii's estimate. For every
$c_0>0$ and $T>0$,
\begin{align}\label{K-exp-fixed-path}
\sup_{z\in C([0,T];\mR^{2d})}
    \mE\exp\left\{
        c_0\int_0^T (|K(z_t-Z_t)|^2+|K(Z_t-z_t)|^2)\,\dif t
    \right\}
    <\infty ,
\end{align}
where $Z$ denotes the solution of \eqref{SDE10}. This follows by applying
Corollary~\ref{Kha1}, uniformly with respect to the deterministic path $z$, on
sufficiently short subintervals and then using the Markov property.
\subsection{\texorpdfstring{Uniform estimates of $\eta^N_t$}{Uniform estimates}}\label{subsec:clt-uniform}
In this subsection we establish the crucial moment estimates for $\eta^N_t$ and some integral functionals by using the variational representation formula \eqref{Var}
and subgaussian moment estimates established in Subsection \ref{sec:subgaussian}.
More precisely,
let $\bar Z^1,\ldots,\bar Z^N$ be independent copies of $Z$, and set
$$
    \bar\mu_t^N:=\frac1N\sum_{i=1}^N\delta_{\bar Z_t^i},
    \qquad
    \bar\eta_t^N:=\sqrt N(\bar\mu_t^N-\mu_t).
$$
Let $F$ be a nonnegative Borel measurable functional over $C([0,T];\mR^{2dN})$.
By the variational representation formula \eqref{Var} of the relative entropy, we have for any $\lambda>0$,
$$    
\mE_{\bP_{[0,T]}^{N,N}} F
\le
    \frac1\lambda
    \left(
        \cH(\bP_{[0,T]}^{N,N}|\bP_{[0,T]}^{\otimes N})
        +
        \log\mE_{\bP_{[0,T]}^{\otimes N}}\exp(\lambda F)
    \right)
$$
Assuming the initial $N$-particle law is exchangeable, this estimate, together
with \eqref{PoC:main}, yields that for some $C_1=C_1(T,d,K)>0$,
\begin{align}\label{eq:clt-entropy-transfer}
    \mE_{\bP_{[0,T]}^{N,N}} F
    &\le
    \frac{C_1}\lambda\left(\cH(\mu_0^{N,N}|\mu_0^{\otimes N})+1+
        \log\mE_{\bP_{[0,T]}^{\otimes N}}\exp(\lambda F)\right).
\end{align}
We use this estimate to prove three important estimates.
The first two estimates are used to show the tightness of $\eta^N_t$ in the space $C([0,T]; \bB^{-3d-3,2}_{2;a,-1})$.
The third estimate will be used to take limits in the proof of Theorem \ref{thm:clt}.
\begin{lemma}\label{Le1}
Assume that $\mu_0^{N,N}$ is exchangeable. For any $T>0$ and $\alpha<-2d$, there is a constant $C>0$ such that for all $t\in[0,T]$ and $N\in\mN$,
    \begin{align}\label{eq:clt-empirical-bound}
        \mE\|\eta_t^N\|_{\bB^{\alpha,2}_{2;a}}^2
        \lesssim_C\cH(\mu_0^{N,N}|\mu_0^{\otimes N})+1.
    \end{align}
\end{lemma}

\begin{proof}
Let $\alpha<-2d$ and $\xi_i:=\delta_{\bar Z^i_t}-\mu_t$. Then by \eqref{easy_emb} and \eqref{measure-emb}
with $p=2$, we have
\begin{align*}
\|\xi_{i}\|_{\bB^{\alpha,2}_{2;a}}\lesssim\|\xi_{i}\|_{\bB^{-2d,\infty}_{2;a}}\leq \|\delta_{\bar Z^i_t}\|_{\bB^{-2d,\infty}_{2;a}}
+\|\mu_t\|_{\bB^{-2d,\infty}_{2;a}}\lesssim 1,
\end{align*}
where the implicit constant only depends on $d,\alpha$.
In particular, $(\xi_{i})_{i=1,\cdots,N}$ are centered and independent  in $\bB^{\alpha,2}_{2;a}$.
Moreover, note that
$$
  \bar\eta_t^N=\frac{1}{\sqrt{N}}\sum_{i=1}^N(\delta_{\bar Z^i_t}-\mu_t)=\frac{1}{\sqrt{N}}\sum_{i=1}^N \xi_i.
$$
Thus, by Lemma~
\ref{lem:abstract-average-holder} with $\mH=\bB^{\alpha,2}_{2;a}$, there are
constants $c_0,C_0>1$, independent of $t,N$, such that
\begin{align}\label{Ep1}
    \mE\exp\left(c_0\|\bar\eta_t^N\|_{\bB^{\alpha,2}_{2;a}}^2\right)= \mE\exp\left(\frac{c_0}{N}\left\|\sum_{i=1}^N \xi_{i}\right\|_{\bB^{\alpha,2}_{2;a}}^2\right)\le C_0<\infty.
\end{align}
Fix $t\in[0,T]$. Applying \eqref{eq:clt-entropy-transfer} to
\[
    F(z^1_\cdot,\cdots,z^N_\cdot)
    =
    \frac1N
    \left\|\sum_{i=1}^N(\delta_{z^i_t}-\mu_t)\right\|_{\bB^{\alpha,2}_{2;a}}^2,
\]
one gets
\begin{align*}
    \mE\|\eta_t^N\|_{\bB^{\alpha,2}_{2;a}}^2
  &  \le C_1
c_0^{-1}
    \left(
        \cH(\mu_0^{N,N}|\mu_0^{\otimes N})+1
        +
        \log\mE\exp(c_0\|\bar\eta_t^N\|_{\bB^{\alpha,2}_{2;a}}^2)
    \right),
\end{align*}
which gives \eqref{eq:clt-empirical-bound} by \eqref{Ep1}.
\end{proof}

\begin{lemma}\label{Le2}
Assume that $\mu_0^{N,N}$ is exchangeable. For any $T>0$ and $\alpha<-2d$, there is a constant $C>0$ 
independent of $N$ such that
    \begin{align}\label{eq:clt-interaction-bound}
        N\int_0^T
        \mE\left\|
        (K*\mu_t^N)\mu_t^N-(K*\mu_t)\mu_t
        \right\|_{\bB^{\alpha,2}_{2;a}}^2\,\dif t\lesssim_C\cH(\mu_0^{N,N}|\mu_0^{\otimes N})+1.
    \end{align}
\end{lemma}
\begin{proof}
We make the following decomposition: 
\begin{align*}
(K*\mu_t^N)\mu_t^N-(K*\mu_t)\mu_t=
(K*(\mu^N_t-\mu_t))\mu^N_t+(K*\mu_t)(\mu^N_t-\mu_t)=:G^N_{1,t}+G^N_{2,t}. 
\end{align*}
We first treat $G_1^N$. To use \eqref{eq:clt-entropy-transfer}, set
$$
    \bar G_{1,t}^N:=(K*(\bar\mu_t^N-\mu_t))\bar\mu_t^N.
$$
For simplicity of notation, write
 $$
 \check{\phi}_j^a:=\cF^{-1}{\phi}_j^a.
 $$
By the definition of $\cR^a_j$,
\begin{align*}
    \cR_j^a\bar G_{1,t}^N
=\frac{1}{N}\sum_{i=1}^N\check{\phi}_j^a(\cdot-\bar{Z}^i_t) U^i_N(t),
\end{align*}
where
$$
U^i_N(t):=\frac1N\sum_{j=1}^N (K(\bar{Z}^i_t-\bar{Z}^j_t)-(K*\mu_t)(\bar{Z}^i_t)).
$$
Note that by Fubini's theorem and Minkowski's inequality,
\begin{align*}
    \|\bar G_{1}^N\|_{L^2_T\bB^{\alpha,2}_{2;a}}^2=\sum_{j=-1}^\infty\int_0^T 2^{2\alpha j}\|\cR_j^a\bar G_{1,t}^N\|_{L^2}^2\dif t\le \sum_{j=-1}^\infty2^{2\alpha j}\|\check{\phi}_j^a\|_{L^2}^2 \frac{1}{N}\sum_{i=1}^N\|U^i_N\|_{L^2([0,T])}^2.
\end{align*}
Since $\alpha<-2d$ and $\|\check{\phi}_j^a\|_{L^2}\lesssim 2^{2dj}$, we have
\begin{align*}
    \sum_{j=-1}^\infty2^{2\alpha j}\|\check{\phi}_j^a\|_{L^2}^2\lesssim  \sum_{j=-1}^\infty2^{2\alpha j+4d j}<\infty. 
\end{align*}
Thus, by \eqref{Eq1}, there is a small constant $c_0>0$ such that
\begin{align*}
    \mE \exp\left(c_0 N \|\bar G_{1}^N\|_{L^2_T\bB^{\alpha,2}_{2;a}}^2\right)
    \le
    \mE \exp\left(c_0 C_d \sum_{i=1}^N\|U^i_N\|_{L^2([0,T])}^2\right)
    \le C_0<\infty,
\end{align*}
with a constant $C_0$ independent of $N$. Indeed, if
\[
    F_{0,T}^{N,i}
    :=
    N\|U_N^i\|_{L^2([0,T])}^2,
\]
then $\sum_i\|U_N^i\|_{L^2([0,T])}^2=N^{-1}\sum_iF_{0,T}^{N,i}$, and
each $F_{0,T}^{N,i}$ has the same law as $F_{0,T}^N$ under
$\bP_{[0,T]}^{\otimes N}$. H\"older's inequality and \eqref{Eq1} therefore
give the displayed bound after reducing $c_0$, if necessary.
Then taking $\lambda=Nc_0$ in \eqref{eq:clt-entropy-transfer} gives that
\begin{align}\label{0607:05}
    \mE\|G_{1}^N\|_{L^2_T\bB^{\alpha,2}_{2;a}}^2
    \le
    \frac{C_1c_{0}^{-1}}{N}\left(
\cH(\mu_0^{N,N}|\mu_0^{\otimes N})+1+\log C_{0}
    \right).
\end{align}

Now we turn to the term $G_2^N$. We still set
$$
    \bar G_{2,t}^N:=(K*\mu_t)(\bar\mu^N_t-\mu_t).
$$
Then $\bar G_{2,t}^N$ is an ordinary i.i.d. empirical fluctuation, and
\begin{align*}
    \mathcal R_j^a\bar G_{2,t}^N
    &=\int_{\mR^{2d}}\check{\phi}_j^a(\cdot-z)(K*\mu_t)(z)(\bar\mu^N_t-\mu_t)(\dif z)=:\frac1N\sum_{i=1}^N V^i_j(t),
\end{align*}
where
$$
V^i_j(t):=\check{\phi}_j^a(\cdot-\bar Z_t^i)(K*\mu_t)(\bar Z_t^i)
        -
        \int_{\mR^{2d}} \check{\phi}_j^a(\cdot-z)(K*\mu_t)(z)\,\mu_t(\dif z).
$$
Note that $(V^i_j)_{i=1,\cdots, N}$ are i.i.d., $\mE V^i_j(t)=0$ and 
\begin{align*}
    \|V^i_j\|_{L^2_TL^2}^2
    \le
    C\|\check{\phi}^a_j\|_{L^2}^2
    \int_0^T
    \left(
        |(K*\mu_t)(\bar{Z}^i_t)|^2
        +
        \langle\mu_t,|K*\mu_t|^2\rangle
    \right)\dif t.
\end{align*}
Since $ \|\check{\phi}^a_j\|_{L^2}\lesssim 2^{2d j}$, Jensen's inequality,
\eqref{K-exp-fixed-path}, and the Krylov estimate imply that there is a
$c_0>0$ such that
\begin{align*}
   \sup_{j,N} \mE\exp\left(c_0 2^{-4d j}  \|V^i_j\|_{L^2_TL^2}^2\right)\le C_{K,T}<\infty.
\end{align*}
Then Lemma \ref{lem:abstract-average-holder} yields constants $c_0,C_0>0$ independent of $N,j$ such that
\begin{align*}
   \mE\exp\left(c_0 2^{-4d j} N \|\mathcal R_j^a\bar G_{2}^N\|_{L^2_TL^2}^2\right)\le C_0,
\end{align*}
which by taking $\lambda=c_0N 2^{-4d j}$ in \eqref{eq:clt-entropy-transfer} implies that
\begin{align*}
    \mE \|\mathcal R_j^a G_{2}^N\|_{L^2_TL^2}^2\le 2^{4d j}\frac{C_1c_0^{-1}}{N}\left(\cH(\mu_0^{N,N}|\mu_0^{\otimes N})+1+\log C_0\right).
\end{align*}
Thus, multiplying both sides by $2^{2\alpha j}$ and summing over $j$ yields
\begin{align*}
    \mE \|G_{2}^N\|_{L^2_T\bB^{\alpha,2}_{2;a}}^2\le \sum_{j=-1}^\infty 2^{2(2d+\alpha)j}\frac{C_1c_0^{-1}}{N}\left(\cH(\mu_0^{N,N}|\mu_0^{\otimes N})+1+\log C_0\right).
\end{align*}
This together with \eqref{0607:05} gives \eqref{eq:clt-interaction-bound}.
\end{proof}

\begin{lemma}\label{Le3}
Assume that $\mu_0^{N,N}$ is exchangeable. For any $T>0$ and any
measurable $\varphi:[0,T]\times\mR^{2d}\to\mR$ satisfying
$\|\nabla_v\varphi\|_{L^\infty([0,T]\times\mR^{2d})}<\infty$,
there is a constant $C=C(T,\varphi)>0$ 
independent of $N$ such that
    \begin{align}\label{eq:clt-quadratic-rem}
        &N\mE\int_0^T
        \left|
        \left\langle
        \nabla_v\varphi_t,
        (K*(\mu_t^N-\mu_t))(\mu_t^N-\mu_t)
        \right\rangle
        \right|\,\dif t\lesssim_C\cH(\mu_0^{N,N}|\mu_0^{\otimes N})+1.
\end{align}
\end{lemma}

\begin{proof}
Fix $T>0$ and such a function $\varphi$.
By \eqref{eq:clt-entropy-transfer}, it suffices to prove that for some $c_\varphi,C_\varphi>0$,
\begin{align}\label{eq:clt-ref-quadratic}
    \sup_{N\ge1}
    \mE\exp\left\{c_\varphi N\int^T_0\left|
        \left\langle
        \nabla_v\varphi_t,
        (K*(\bar\mu_t^N-\mu_t))(\bar\mu_t^N-\mu_t)
        \right\rangle
        \right|\,\dif t\right\}
    \le C_\varphi .
\end{align}
Below,
we use Lemma \ref{lem:clt-centered-double-sum} with $\mH:=L^2([0,T])$ to show the above estimate.
For this aim, we define for any $t\in[0,T]$ and $z,z'\in\mR^{2d}$,
$$
    \Phi_t(z,z'):=\nabla_v\varphi_t(z)\cdot K(z-z'),
$$
and
\begin{align}\label{eq:clt-centered-kernel}
\Gamma_t(z,z')
    &:=
    \Phi_t(z,z')
    -\mu_t(\Phi_t(z,\cdot))-\mu_t(\Phi_t(\cdot,z'))
    +(\mu_t\otimes\mu_t)(\Phi_t).
\end{align}
With the above notations, we have
$$
    \left\langle
    \nabla_v\varphi_t,(K*(\bar\mu_t^N-\mu_t))(\bar\mu_t^N-\mu_t)
    \right\rangle
    =\frac1{N^2}\sum_{i,j=1}^N\Gamma_t(\bar Z^i_t,\bar Z^j_t).
$$
For any $z,z'\in C([0,T];\mR^{2d})$, we set
$$
    h(z,z')(t):=\Gamma_t(z_t,z'_t).
$$
Clearly, for each $t\in[0,T]$,
$$
    \mE h(z,\bar Z)(t)
    =
    \mE h(\bar Z,z')(t)
    =0.
$$
By Krylov estimate \eqref{Kry10}, it is easy to see that
$$
h(z,\bar Z),\ h(\bar Z,z')\in L^2(\Omega\times[0,T]),
$$
and as in the proof of Corollary~\ref{Kha1},  there is a small constant $a_\varphi>0$ such that
$$
    \sup_{z\in C([0,T];\mR^{2d})} \mE\exp\bigl(a_\varphi
    \|h(z,\bar Z)\|_{L^2([0,T])}^2\bigr)
    +
    \sup_{z'\in C([0,T];\mR^{2d})} \mE\exp\bigl(a_\varphi
    \|h(\bar Z,z')\|_{L^2([0,T])}^2\bigr)
    \le C_\varphi .
$$
Thus, by Lemma~\ref{lem:clt-centered-double-sum} with $\mH=L^2([0,T])$,  we have for some $c_\varphi>0$,
\begin{align}\label{Gh1}
    \sup_{N\ge2}
    \mE\exp\left(
        \frac{c_\varphi}{N}
        \Big\|
        \sum_{i\ne j}
        h(\bar Z^i,\bar Z^j)
        \Big\|_{L^2([0,T])}
    \right)<\infty.
\end{align}
On the other hand,  by Khasminskii's estimate, we also have
for some $a_\varphi, C_\varphi>0$,
$$
    \mE\exp\bigl(a_\varphi\|\Gamma_{\cdot}(\bar Z_\cdot,\bar Z_\cdot)\|_{L^2([0,T])}^2\bigr)
    \le C_\varphi,
$$
which in turn implies that for some $c_\varphi>0$,
$$
    \mE\exp\left(c_\varphi
        \|\Gamma_{\cdot}(\bar Z_\cdot,\bar Z_\cdot)\|_{L^1([0,T])}
    \right)
    \le C_\varphi .
$$
Finally, by convexity of the exponential, we have
\begin{align*}
\mE\exp\left(    \frac{c_\varphi}N\Big\|\sum_{i=1}^N
        \Gamma_{\cdot}(\bar Z^i_\cdot,\bar Z^i_\cdot)\Big\|_{L^1([0,T])}
    \right)
&\le
\mE\exp\left(    \frac{c_\varphi}N\sum_{i=1}^N
        \|\Gamma_{\cdot}(\bar Z^i_\cdot,\bar Z^i_\cdot)\|_{L^1([0,T])}
    \right)\\
&    \le
    \frac1N\sum_{i=1}^N
    \mE\exp\left(
        c_\varphi\|\Gamma_{\cdot}(\bar Z^i_\cdot,\bar Z^i_\cdot)\|_{L^1([0,T])}
    \right)
    \le C_\varphi .
\end{align*}
This together with \eqref{Gh1} gives \eqref{eq:clt-ref-quadratic}.
\end{proof}

\subsection{Proof of Theorem \ref{thm:clt}}\label{subsec:clt-proof}

\begin{proof}[Proof of Theorem \ref{thm:clt}]
The proof is organized into three parts: uniform estimates, tightness, and identification of limit points. The uniform estimates established in Lemmas~\ref{Le1}--\ref{Le2} provide $L^2$-bounds for $\eta_t^N$ and for the nonlinear interaction term, uniformly in $N$. Together with the martingale estimates from the Hilbert-space BDG inequality, these bounds imply tightness of $(\eta^N, M^N)$ in $C([0,T];\bB^{\gamma-3,2}_{2;a,\ell-1})\times C([0,T];\bB^{\gamma-1,2}_{2;a,\ell})$ for $\gamma<\alpha<-2d$ and $\ell<0$. Identification of limit points uses the expansion \eqref{eq:clt-interaction-expansion}: Lemma~\ref{Le3} shows that the quadratic remainder vanishes, while the linear terms converge by tightness and localization. The stable martingale central limit theorem identifies the limiting noise as a Gaussian martingale with covariance \eqref{eq:clt-noise-covariance}. Finally, uniqueness of the limiting martingale problem implies convergence of the full sequence.
Below we fix $\gamma<\alpha<-2d$.

\emph{Step 1: uniform estimates and tightness.}
Note that by \eqref{eq:clt-finite-equation}
\begin{align}\label{Eq2}
\begin{split}
\eta_t^N
    &=
\eta_0^N
    +
    \int_0^t(\Delta_v-v\cdot\nabla_x)\eta_s^N\,\dif s
-\sqrt N\int_0^t
        \div_v((K*\mu_s^N)\mu^N_s-(K*\mu_s)\mu_s)\,\dif s-    M_t^{N},
 \end{split}
\end{align}
where
$$
M_t^N=\frac{\sqrt{2}}{\sqrt N}\sum_{i=1}^N\int_0^t\nabla_v\delta_{Z_s^{N,i}}\cdot\dif W_s^i.
$$
Since $\mu_0^{N,N}=\mu_0^{\otimes N}$, by Bernstein's inequality \eqref{eq:clt-lp-gradient} and Lemma \ref{Le1}, we have
\begin{align}\label{0615:00}
    \sup_{N\ge1}\sup_{s\in[0, T]}\mE\|(\Delta_v-v\cdot\nabla_x)\eta_s^N\|_{\bB^{\alpha-3,2}_{2;a,-1}}^2\lesssim
    \sup_{N\ge1}\sup_{s\in[0, T]}\mE\|\eta_s^N\|_{\bB^{\alpha,2}_{2;a}}^2<\infty,
\end{align}
and by Lemma \ref{Le2},
\begin{align}\label{eq:clt-uniform-interaction}
    \sup_{N\ge1}
    \int_0^T
    \mE\left\|\sqrt N
        \div_v
        \left(
        (K*\mu_t^N)\mu_t^N-(K*\mu_t)\mu_t
        \right)
    \right\|_{\bB^{\alpha-1,2}_{2;a}}^2\,\dif t
    <\infty .
\end{align}
On the other hand, since by \eqref{easy_emb} and \eqref{measure-emb}, $\nabla_v\delta_{Z_s^{N,i}}\in \bB^{\alpha-1,2}_{2;a}$ ,
by the Hilbert space-valued BDG inequality, we have for every $p\geq1$ and $0\le s\le t<\infty$,
\begin{align*}
\begin{split}
    \mE\|M_t^N-M_s^N\|_{\bB^{\alpha-1,2}_{2;a}}^{2p}\lesssim \frac{1}{N^p}\mE\left(\sum_{i=1}^N
    \int_s^t\|\nabla_v\delta_{Z^{N,i}_r}\|^2_{\bB^{\alpha-1,2}_{2;a}}\dif r\right)^p
    \lesssim |t-s|^p,
    \end{split}
\end{align*}
which implies by Kolmogorov's criterion that for any $\theta\in(0,1)$,
\begin{align}\label{eq:clt-mart-holder}
\begin{split}
     \sup_{N\ge1}
    \mE\left[\sup_{0\le s<t\le T}\frac{\|M_t^N-M_s^N\|_{\bB^{\alpha-1,2}_{2;a}}^{2}}{|t-s|^\theta}\right]<\infty.
    \end{split}
\end{align}
Combining \eqref{Eq2}, \eqref{0615:00}, \eqref{eq:clt-uniform-interaction} and
\eqref{eq:clt-mart-holder}, we obtain that for $\theta\in(0,1)$,
\begin{align}\label{eq:clt-time-modulus}
    \sup_{N\ge1}
    \mE\left(
    \sup_{0\le s<t\le T}
    \frac{\|\eta_t^N-\eta_s^N\|_{\bB^{\alpha-3,2}_{2;a,-1}}^2}{|t-s|^\theta}\right)
    <\infty.
\end{align}
Since $\gamma<\alpha$ and $\ell<0$, by \cite[Lemma A.3]{HZZZ22}, we have the compact embeddings 
$$
    \bB^{\alpha-3,2}_{2;a,-1}(\mR^{2d})
    \hookrightarrow
    \bB^{\gamma-3,2}_{2;a,\ell-1}(\mR^{2d}),\ \ \bB^{\alpha-1,2}_{2;a}(\mR^{2d})
    \hookrightarrow
    \bB^{\gamma-1,2}_{2;a,\ell}(\mR^{2d}).
$$
Thus, by \eqref{eq:clt-empirical-bound} and \eqref{eq:clt-time-modulus}, we have
$$
\mbox{$(\eta^N, M^N)_{N\ge1}$ is tight in $C([0,T]; \bB^{\gamma-3,2}_{2;a,\ell-1}(\mR^{2d}))\times C([0,T]; \bB^{\gamma-1,2}_{2;a,\ell}(\mR^{2d}))$}. 
$$

\emph{Step 2: identification of subsequential limits.}
Let $(\eta^{N_k},M^{N_k})$ be a convergent subsequence. 
By Skorohod's representation, on a new probability
space, we may suppose that
\begin{align}\label{Lim1}
    (\eta^{N_k}, M^{N_k})\to(\eta, M)
\hbox{ in }C([0,T]; \bB^{\gamma-3,2}_{2;a,\ell-1}(\mR^{2d}))\times C([0,T]; \bB^{\gamma-1,2}_{2;a,\ell}(\mR^{2d})),\ \ a.s.
\end{align}
The standard empirical-measure central limit theorem gives
$\eta_0^{N_k}\Rightarrow\eta_0$, where $\eta_0$ is the centered Gaussian field
with the covariance stated in the theorem. Moreover, by \eqref{eq:clt-empirical-bound} and lower semicontinuity,
$$
    \eta\in L^\infty([0,T];L^2(\Omega;\bB^{\alpha,2}_{2;a})).
$$
Now we want to take limits for both sides of \eqref{eq:clt-finite-equation}.
Let $\varphi\in C^1(\mR_+; C^\infty_c(\mR^{2d}))$.
First, the convergence in \eqref{Lim1} implies
$$
\int_0^t\langle\eta_s^{N_k},(\p_s+\Delta_v+v\cdot\nabla_x)\varphi_s\rangle\,\dif s\to
\int_0^t\langle\eta_s,(\p_s+\Delta_v+v\cdot\nabla_x)\varphi_s\rangle\,\dif s.
$$
Note that
\begin{align}\label{eq:clt-interaction-expansion}
&\sqrt {N_k}\left[
\langle(K*\mu_t^{N_k})\mu_t^{N_k},\nabla_v\varphi_t\rangle
-
\langle(K*\mu_t)\mu_t,\nabla_v\varphi_t\rangle
\right]\notag\\
&\quad =
\langle(K*\mu_t)\eta_t^{N_k},\nabla_v\varphi_t\rangle
+
\langle(K*\eta_t^{N_k})\mu_t,\nabla_v\varphi_t\rangle\no\\
&\qquad+
\sqrt {N_k}\langle (K*(\mu_t^{N_k}-\mu_t))(\mu_t^{N_k}-\mu_t),\nabla_v\varphi_t\rangle\nonumber\\
&\quad =
\langle\sA_{\mu_t}\eta_t^{N_k},\nabla_v\varphi_t\rangle
+\sqrt {N_k}\langle (K*(\mu_t^{N_k}-\mu_t))(\mu_t^{N_k}-\mu_t),\nabla_v\varphi_t\rangle .
\end{align}
By Lemma \ref{Le3}, we have
\begin{align}\label{eq:clt-remainder-vanishes}
    \sqrt N\,
    \int_0^T
    \mE\left|
    \left\langle
        \nabla_v\varphi_s,
        (K*(\mu_s^N-\mu_s))(\mu_s^N-\mu_s)
    \right\rangle
    \right|\,\dif s
    \longrightarrow0 .
\end{align}
Indeed, since $\mu_0^{N,N}=\mu_0^{\otimes N}$, Lemma~\ref{Le3} gives a
uniform bound on
$N\int_0^T\mE|\langle\nabla_v\varphi,
(K*(\mu_s^N-\mu_s))(\mu_s^N-\mu_s)\rangle|\,\dif s$, and
\eqref{eq:clt-remainder-vanishes} follows by division by $\sqrt N$.

We next pass to the limit in the linear terms. 
By Theorem~\ref{lem:mu-besov-apriori}, Lemma \ref{Le10}, and a standard localization argument,
the convergence \eqref{Lim1} gives
\[
    \int_\varepsilon^t \langle\sA_{\mu_s}\eta_s^{N_k},\nabla_v\varphi_s\rangle\,\dif s
    \longrightarrow
    \int_\varepsilon^t \langle\sA_{\mu_s}\eta_s,\nabla_v\varphi_s\rangle\,\dif s
    \qquad\text{a.s.}
\]
On the remaining interval $(0,\varepsilon)$, Lemma~\ref{Le10},
\eqref{eq:mu-besov-smoothing}, and the uniform bound
\eqref{eq:clt-empirical-bound} imply, for some $\delta\in(0,1)$,
\[
    \sup_k
    \mE\int_0^\varepsilon
    |\langle\sA_{\mu_s}\eta_s^{N_k},\nabla_v\varphi_s\rangle|\,\dif s
    +
    \mE\int_0^\varepsilon
    |\langle\sA_{\mu_s}\eta_s,\nabla_v\varphi_s\rangle|\,\dif s
    \lesssim
    \int_0^\varepsilon s^{-\delta}\,\dif s
    \stackrel{\varepsilon\downarrow0}{\longrightarrow}0 .
\]
Thus, as $k\to\infty$,
\begin{align}\label{eq:clt-linear-term-1}
    \int_0^t
\langle\sA_{\mu_s}\eta_s^{N_k},\nabla_v\varphi_s\rangle\,\dif s\to
\int_0^t
\langle\sA_{\mu_s}\eta_s,\nabla_v\varphi_s\rangle \,\dif s.
\end{align}
It remains to identify the noise and its joint law with the initial field.
Note that the quadratic variation process of continuous martingale
$M^{N,\varphi}$ is given by
\begin{align}\label{eq:clt-finite-cov}
    \langle M^{N,\varphi}\rangle_t
    =
    2\int_0^t
    \langle\mu_s^N,|\nabla_v\varphi_s|^2\rangle\,\dif s .
\end{align}
Since $\mu_s^N-\mu_s=N^{-1/2}\eta_s^N$, \eqref{eq:clt-empirical-bound} implies
that the right-hand side converges in probability, uniformly in $t$, to
\[
    2\int_0^t
    \langle\mu_s,|\nabla_v\varphi_s|^2\rangle\,\dif s .
\]
The stable martingale central limit theorem for continuous martingales then
gives, jointly with the initial empirical field,
\[
    (\eta_0^N,M^N)\Rightarrow(\eta_0,M),
\]
where $M$ is a centered Gaussian martingale field with covariance
\eqref{eq:clt-noise-covariance}. Since the Brownian motions are independent of
the i.i.d. initial variables and the limiting bracket is deterministic, this
stable convergence also gives that $\eta_0$ and $M$ are independent. Therefore
every subsequential limit is a martingale solution in the sense of
Definition~\ref{def:clt-mart-sol} with the joint law of $(\eta_0,M)$ specified
above.

\emph{Step 3: uniqueness.}
By linearity, uniqueness in law for the martingale problem follows once the
homogeneous equation with zero initial value and zero martingale part has only
the zero solution. Thus it suffices to show that $\eta=0$ is the only such
solution of
\begin{align}\label{eq:clt-homogeneous}
\p_t\eta_t
    =
    (\Delta_v-v\cdot\nabla_x)\eta_t-\div_v\bigl(\sA_{\mu_t}\eta_t\bigr).
\end{align}
By Duhamel's formula, we have
\begin{align}\label{eq:clt-duhamel-u}
\eta_t
&=-\int^t_0P_{t-s}^\ast\div_v\bigl(\sA_{\mu_s}\eta_s\bigr)\dif s
=-\int_0^tP_{t-s}^\ast\div_v
\bigl((K*\mu_s)\eta_s+\mu_s(K*\eta_s)\bigr)\,\dif s.
\end{align}
Let $\alpha\in(1-\beta-2d,-2d)$.
Fix $\varepsilon\in(0,\beta-1)$. By \eqref{easy_emb} and
Theorem~\ref{lem:mu-besov-apriori}, for $s\in(0,T]$,
\begin{align*}
\|\mu_s\|_{\bB^{2d+3-\beta,2}_{2;a}}
&\lesssim s^{-\frac{2-\beta+\varepsilon}{2}},
&
\|\mu_s\|_{\bB^{2d+1,2}_{2;a}}
&\lesssim s^{-\frac{\varepsilon}{2}}.
\end{align*}
By Lemma \ref{lem:kinetic-besov-semigroup}, \eqref{Lm1}, \eqref{Lm2} and Lemma \ref{lem:mu-besov-apriori}, we have
\begin{align*}
\|\eta_t\|_{\bB^{\alpha,2}_{2;a}}
&\lesssim \int^t_0\left[(t-s)^{-\frac12}\|(K*\mu_s)\eta_s\|_{\bB^{\alpha,2}_{2;a}}+(t-s)^{-\frac{3-\beta}2}\|(K*\eta_s)\mu_s\|_{\bB^{\alpha+\beta-2,2}_{2;a}}\right]\dif s\\
&\lesssim \int^t_0\left[(t-s)^{-\frac12}s^{-\frac{2-\beta+\varepsilon}{2}}+(t-s)^{-\frac{3-\beta}2}s^{-\frac{\varepsilon}{2}}\right]\|\eta_s\|_{\bB^{\alpha,2}_{2;a}}\dif s.
\end{align*}
The restriction $\varepsilon<\beta-1$ makes the two endpoint singularities integrable.
By Gr\"onwall's inequality of Volterra type, we obtain $\|\eta_t\|_{\bB^{\alpha,2}_{2;a}}=0$ for all $t\in[0,T]$. This proves uniqueness of the homogeneous problem and hence uniqueness in law for the martingale problem. The proof of Theorem~\ref{thm:clt} is complete.
\end{proof}

\subsection{Normal approximation}
In this subsection we derive a smooth-test normal-approximation rate and a
Berry--Esseen bound for finite-dimensional projections.
We first study the following backward Kolmogorov equation: for fixed $t>0$,
\begin{align}\label{eq:clt-backward-Q}
    \partial_s f_s
    +(\Delta_v+v\cdot\nabla_x)f_s+\sA^*_{\mu_s}\cdot\nabla_v f_s=0,
    \qquad
    f_t=\varphi ,
\end{align}
where $\mu_s$ is the law of solution $Z_s$ to SDE \eqref{SDE10} and $\sA_{\mu_s}$ is defined by \eqref{Def1}. We establish the following regularity estimate.
\begin{lemma}\label{Le49}
Under Assumption \ref{ass:clt}, for every $T>0$ there exists a constant
$C_T>0$ such that, for all $0<t\le T$ and
$\varphi\in \bB^{2d+2,2}_{2;a}$, PDE \eqref{eq:clt-backward-Q} has a unique
solution satisfying
    \begin{align}\label{supply:00}
        \sup_{s\in[0,t]}\|f_s\|_{\bB^{2d+2,2}_{2;a}}\le C_T
        \|\varphi\|_{\bB^{2d+2,2}_{2;a}}. 
    \end{align}
	To display the dependence on the terminal time $t$ and terminal value $\varphi$, we write
$$
Q_{s,t}\varphi=f_s.
$$
\end{lemma}
\begin{proof}
Note that by Duhamel's formula,
\begin{align*}
    f_s=P_{t-s}\varphi +\int_s^t P_{r-s}(\sA^*_{\mu_r}\cdot\nabla_v f_r)\dif r.
\end{align*}
Let $\widetilde K(z):=K(-z)$.
Fix $\varepsilon\in(0,\beta-1)$. By \eqref{easy_emb} and
Theorem~\ref{lem:mu-besov-apriori}, for $r\in(0,t]$,
\begin{align*}
\|\mu_r\|_{\bB^{2d+3-\beta,2}_{2;a}}
&\lesssim r^{-\frac{2-\beta+\varepsilon}{2}},
&
\|\mu_r\|_{\bB^{2d+1,2}_{2;a}}
&\lesssim r^{-\frac{\varepsilon}{2}}.
\end{align*}
Then Lemma \ref{Le23} gives
\begin{align*}
\|f_s\|_{\bB^{2d+2,2}_{2;a}}\lesssim&\|\varphi\|_{\bB^{2d+2,2}_{2;a}}+\int_s^t (r-s)^{-\frac12}\|(K*\mu_r)\cdot\nabla_v f_r\|_{\bB^{2d+1,2}_{2;a}}\dif r\\
    &+\int_s^t (r-s)^{-\frac{3-\beta}{2}}\|\widetilde K*(\mu_r\nabla_v f_r)\|_{\bB^{2d+\beta-1,2}_{2;a}}\dif r\\
    \lesssim&\|\varphi\|_{\bB^{2d+2,2}_{2;a}}+\int_s^t (r-s)^{-\frac12}\|K\|_{\bB^{\beta-2,\infty}_{2;a}}\|\mu_r\|_{\bB^{2d+3-\beta,2}_{2;a}}\|f_r\|_{\bB^{2d+2,2}_{2;a}}\dif r\\
    &+\int_s^t (r-s)^{-\frac{3-\beta}{2}}\|K\|_{\bB^{\beta-2,\infty}_{2;a}}\|\mu_r\|_{\bB^{2d+1,2}_{2;a}}\|f_r\|_{\bB^{2d+2,2}_{2;a}}\dif r\\
    &\lesssim \|\varphi\|_{\bB^{2d+2,2}_{2;a}}
    +\int_s^t (r-s)^{-\frac12}r^{-\frac{2-\beta+\varepsilon}{2}}
        \|f_r\|_{\bB^{2d+2,2}_{2;a}}\dif r\\
    &\quad+\int_s^t (r-s)^{-\frac{3-\beta}{2}}r^{-\frac{\varepsilon}{2}}
        \|f_r\|_{\bB^{2d+2,2}_{2;a}}\dif r,
\end{align*}
The restrictions $\varepsilon<\beta-1$ and $\beta\in(1,2)$ make
both Volterra kernels integrable at their endpoints. Picard iteration in
$C([0,t];\bB^{2d+2,2}_{2;a})$, followed by the corresponding generalized
Gr\"onwall inequality, gives existence, uniqueness, and the uniform estimate
\eqref{supply:00}.
\end{proof}
We first prove a smooth-test normal-approximation bound.
\begin{theorem}\label{thm:smooth-weak-rate}
Under the assumptions of Theorem \ref{thm:clt}, for any $m\in\N$ and
$\varphi_1,\ldots,\varphi_m\in C_c^\infty(\mR^{2d})$, there exists a
constant $C=C(T,m,\varphi_1,\ldots,\varphi_m,K,\mu_0)>0$, independent of $N$,
such that, for every $F\in C_b^2(\R^m)$,
\begin{align}\label{eq:weak-rate}
\begin{split}
&\sup_{t\in[0,T]}\abs{\E F\bigl(\langle{\eta_t^N},{\varphi_1}\rangle,\ldots,\langle{\eta_t^N},{\varphi_m}\rangle\bigr)
    -\E F\bigl(\langle{\eta_t},{\varphi_1}\rangle,\ldots,\langle{\eta_t},{\varphi_m}\rangle\bigr)}\\
    &\qquad\qquad\le C N^{-1/2}\left(\norm{\nabla F}_{L^\infty(\R^m)}+
\|{\nabla^2F}\|_{L^\infty(\R^m)}\right).
\end{split}
\end{align}
\end{theorem}
\begin{proof}
Fix $t\le T$ and $\varphi_1,\ldots,\varphi_m\in C_c^\infty(\mR^{2d})$. 
For each $i=1,\cdots, m$, let $f_s^i$ solve the following backward PDE:
\begin{equation}\label{eq:backward-cancel}
    \partial_s f_s^i+(\Delta_v+v\cdot\nabla_x)f^i_s+\sA_{\mu_s}^*\cdot\nabla_v f_s^i=0,\ s\in[0,t],\ \ f_t^i=\varphi_i.
\end{equation}
By Lemma \ref{Le49} and Bernstein's inequality,
\begin{align}\label{Dh1}
\sup_{s\in[0,t]}\|\nabla_v f^i_s\|_{L^\infty}\lesssim
\|\varphi_i\|_{\bB^{2d+2,2}_{2;a}}.
\end{align}
Define the finite-dimensional processes
$$
    X_s^{N,i}:=\ip{\eta_s^N}{f_s^i},
    \qquad
    X_s^i:=\ip{\eta_s}{f_s^i},
    \qquad 0\le s\le t.
$$
By It\^o's formula, we have
\begin{equation}\label{eq:XN-dynamics}
    \dif X_s^{N,i}=b_s^{N,i}\dif s+\dif M_s^{N,f^i},
    \end{equation}
where 
$$
b_s^{N,i}:=\sqrt N\,\ip{(K*(\mu_s^N-\mu_s))(\mu_s^N-\mu_s)}{\nabla_v  f_s^i},
$$
Moreover, by \eqref{eq:linearized-clt-weak} we also have
\begin{equation}\label{eq:X-dynamics}
    \dif X_s^i=\dif M_s^{f^i}.
\end{equation}
By \eqref{Dh1} and Lemma \ref{Le3}, we have
\begin{equation}\label{eq:b-small}
    \int_0^t\E\abs{b_s^{N,i}}\dif s
    \le C_{T,\varphi_i}N^{-1/2}.
\end{equation}
Note that the quadratic covariations are
$$
    \dif\langle M^{N,f^i}, M^{N,f^j}\rangle_s
    =\cC_s^{N,ij}\dif s,
    \quad
    \cC_s^{N,ij}:=2\ip{\mu_s^N}{\nabla_v  f_s^i\cdot\nabla_v  f_s^j},
$$
and
$$
    \dif\langle M^{f^i}, M^{f^j}\rangle_s
    =\cC_s^{ij}\dif s,
    \quad
    \cC_s^{ij}:=2\ip{\mu_s}{\nabla_v  f_s^i\cdot\nabla_v  f_s^j}.
$$
For the following estimate, fix $\alpha=-2d-1$.
Therefore,
\begin{align*}
    \int_0^t\E\abs{\cC_s^{N,ij}-\cC_s^{ij}}\dif s
    &\le 2N^{-1/2}\int_0^t
        \E\abs{\ip{\eta_s^N}{\nabla_v  f_s^i\cdot\nabla_v  f_s^j}}\dif s \notag \\
    &\le 2N^{-1/2}\sup_{s\in[0,t]}
        \norm{\nabla_v  f_s^i\cdot\nabla_v  f_s^j}_{\bB^{-\alpha,2}_{2;a}}
        \int_0^t\E\norm{\eta_s^N}_{\bB^{\alpha,2}_{2;a}}\dif s.
\end{align*}
By the product estimate, Bernstein's inequality in Lemma \ref{Le22},
and \eqref{supply:00},
\begin{align*}
    \lVert \nabla_vf^i_s\cdot \nabla_v f^j_s \rVert_{\bB^{2d+1,2}_{2;a}} &\lesssim \lVert \nabla_v f^i_s \rVert_{\bB^{2d+1,2}_{2;a}} \lVert \nabla_v f^j_s \rVert_{\infty} + \lVert \nabla_v f^i_s \rVert_{\infty} \lVert \nabla_v f^j_s \rVert_{\bB^{2d+1,2}_{2;a}} \\
    &\lesssim \lVert f^i_s \rVert_{\bB^{2d+2,2}_{2;a}} \lVert \nabla_v f^j_s \rVert_{\bB^{2d,1}_{2;a}} + \lVert \nabla_v f^i_s \rVert_{\bB^{2d,1}_{2;a}} \lVert f^j_s \rVert_{\bB^{2d+2,2}_{2;a}} \\
    &\lesssim \lVert f^i_s \rVert_{\bB^{2d+2,2}_{2;a}} \lVert f^j_s \rVert_{\bB^{2d+2,2}_{2;a}} \\
    &\lesssim \lVert \varphi_i \rVert_{\bB^{2d+2,2}_{2;a}} \lVert \varphi_j \rVert_{\bB^{2d+2,2}_{2;a}}.
\end{align*}
Hence, by \eqref{eq:clt-empirical-bound},
\begin{align}\label{eq:bracket-small}
    \int_0^t\E\abs{\cC_s^{N,ij}-\cC_s^{ij}}\dif s
\le C_{T,\varphi_i,\varphi_j}N^{-1/2}.
\end{align}

Next, we let $\cC_s=(\cC_s^{ij})_{1\le i,j\le m}$ and
$\cC_s^N=(\cC_s^{N,ij})_{1\le i,j\le m}$. The matrix $\cC_s$ is symmetric and non-negative. Choose a measurable square root $\sigma_s$ with $\sigma_s\sigma_s^\top=\cC_s$. Let $B$ be a standard $m$-dimensional Brownian motion and define, for $0\le s\le t$ and $x\in\R^m$,
\begin{equation}\label{eq:u-definition}
    u_s(x):=\E\left[F\left(x+\int_s^t\sigma_r\dif B_r\right)\right].
\end{equation}
Then $u$ solves
\begin{equation}\label{eq:finite-kolmogorov}
    \partial_su_s+\tfrac12\tr(\cC_s\cdot\nabla^2 u_s)=0,
    \quad u_t=F.
\end{equation}
Moreover, differentiating under the expectation in \eqref{eq:u-definition} yields
\begin{equation}\label{eq:u-derivative-bounds}
    \norm{\nabla u}_{L^\infty([0,t]\times\R^m)}\le
    \norm{\nabla F}_{L^\infty(\R^m)},
    \quad
    \norm{\nabla^2u}_{L^\infty([0,t]\times\R^m)}\le
    \norm{\nabla^2F}_{L^\infty(\R^m)}.
\end{equation}

Now, write 
$$
X_s^N=(X_s^{N,1},\ldots,X_s^{N,m}),\ \ X_s=(X_s^1,\ldots,X_s^m).
$$
Applying Ito's formula to $u_s(X_s^N)$ and using \eqref{eq:finite-kolmogorov} gives
\begin{align}\label{eq:ito-N}
    \E F(X_t^N)
    &=\E u_0(X_0^N)
      +\E\int_0^t\nabla u_s(X_s^N)\cdot b_s^N\dif s+\frac12\E\int_0^t
      \tr\Big(\nabla^2u_s(X_s^N)\cdot
      \bigl(\cC_s^{N}-\cC_s\bigr)\Big)\dif s,
\end{align}
where $b_s^N=(b_s^{N,1},\ldots,b_s^{N,m})$.
For the limiting process, by \eqref{eq:X-dynamics}, we have
\begin{equation}\label{eq:ito-limit}
    \E F(X_t)=\E u_0(X_0).
\end{equation}
Subtracting \eqref{eq:ito-limit} from \eqref{eq:ito-N}, we obtain
\begin{equation}\label{eq:three-terms}
    \abs{\E F(X_t^N)-\E F(X_t)}
    \le I_0^N+I_1^N+I_2^N,
\end{equation}
where
$$
    I_0^N:=\abs{\E u_0(X_0^N)-\E u_0(X_0)},
$$
$$
    I_1^N:=\E\int_0^t\abs{\nabla u(s,X_s^N)\cdot b_s^N}\dif s,
$$
and
$$
    I_2^N:=\frac12\E\int_0^t
    \abs{\tr\Big(\nabla^2u_s(X_s^N)\cdot
      \bigl(\cC_s^{N}-\cC_s\bigr)\Big)}\dif s.
$$
The vector $X_0^N$ is a normalized sum of i.i.d. centered bounded
random vectors, whereas $X_0$ is the centered Gaussian vector with the same
covariance. The uniform estimate \eqref{supply:00} makes the bounds on these
summands uniform over $t\in[0,T]$. Hence the standard multivariate
smooth-test Berry--Esseen estimate \cite{Raic19}, applied on the range of the
covariance matrix when it is singular, and \eqref{eq:u-derivative-bounds} give
\begin{equation}\label{eq:I0-bound}
    I_0^N\le C_{T,\boldsymbol\varphi}\|\nabla F\|_\infty N^{-1/2}.
\end{equation}
Second, by \eqref{eq:u-derivative-bounds} and \eqref{eq:b-small},
\begin{align}\label{eq:I1-bound}
    I_1^N
    \le \norm{\nabla F}_{L^\infty}
        \int_0^t\E\abs{b_s^{N}}\dif s \le C_{T,\boldsymbol\varphi}\norm{\nabla F}_{L^\infty}N^{-1/2}.
\end{align}
Third, using \eqref{eq:u-derivative-bounds} and \eqref{eq:bracket-small},
\begin{align}\label{eq:I2-bound}
    I_2^N
    \le \frac12\norm{\nabla^2F}_{L^\infty}
\int_0^t
        \E\|\cC_s^{N}-\cC_s\|_{\mathrm F}\dif s
    \le C_{T,\boldsymbol\varphi}\norm{\nabla^2F}_{L^\infty}N^{-1/2}.
\end{align}
Combining \eqref{eq:three-terms}, \eqref{eq:I0-bound}, \eqref{eq:I1-bound}, 
and \eqref{eq:I2-bound}, we obtain \eqref{eq:weak-rate}. 
\end{proof}

The following lemma is a direct consequence of Lemma \ref{Le49} and the definition of solution to SPDE \eqref{eq:linearized-clt-formal}.
\begin{lemma}\label{prop:clt-ou-covariance}
Under the assumptions of Theorem~\ref{thm:clt}, for each $\phi\in C^\infty_c(\mR^{2d})$ and $t\geq 0$,
\begin{align}\label{eq:clt-eta-variance}
   \langle\eta_t,\phi\rangle\sim \mathcal N\left(0,
    \langle\mu_0,|Q_{0,t}\phi|^2\rangle
    -
    \langle\mu_0,Q_{0,t}\phi\rangle^2+
    2\int_0^t
        \langle\mu_s,|\nabla_v Q_{s,t}\phi|^2\rangle\,\dif s\right).
\end{align}
\end{lemma}

\begin{proof}
Let $f$ solve the backward Kolmogorov equation \eqref{eq:clt-backward-Q} with final value $f_t=\phi$. 
Applying \eqref{eq:linearized-clt-weak} to $\varphi_s=f_s=Q_{s,t}\phi$ yields
$$
\langle\eta_t,\phi\rangle
    =
    \langle\eta_0,f_0\rangle
    +M_t^f.
$$
The asserted variance follows from the covariance of $\eta_0$, the covariance
of $M^f$, and their independence.
\end{proof}

We finally give a Berry--Esseen bound for finite-dimensional projections of $\eta^N_t$.
\bt[Berry--Esseen estimate]\label{thm:berry-esseen}
Let $m\in\mathbb N$ and 
$\varphi_1,\ldots,\varphi_m\in C^\infty_c(\mR^{2d})$. For $t\in[0,T]$, define
\begin{equation*}
Y_t^N
:=
\big(
\langle \eta_t^N,\varphi_1\rangle,\ldots,
\langle \eta_t^N,\varphi_m\rangle
\big),\quad Y_t
:=
\big(
\langle \eta_t,\varphi_1\rangle,\ldots,
\langle \eta_t,\varphi_m\rangle
\big),
\end{equation*}
and for each $k=1,\ldots,m$,
\begin{equation*}
\sigma_{t,k}^2
:=
    \langle\mu_0,|Q_{0,t}\varphi_k|^2\rangle
    -
    \langle\mu_0,Q_{0,t}\varphi_k\rangle^2
    +
    2\int_0^t
        \langle\mu_r,|\nabla_v Q_{r,t}\varphi_k|^2\rangle\,\dif r
    >0.
\end{equation*}
Then there is a constant $C>0$, depending on
$t,m,\varphi_1,\ldots,\varphi_m,K$, $\mu_0$, and
$\sum_{k=1}^m\sigma_{t,k}^{-1}$, such that
\begin{align}\label{Berry-Esseen-m}
\sup_{x\in\mR^m}
\left|
\mathbb P(Y_t^N\le x)-\mathbb P(Y_t\le x)
\right|
\le C N^{-1/6},
\end{align}
where for $x=(x_1,\cdots,x_m)\in\mR^m$ and $y=(y_1,\cdots,y_m)\in\mR^m$,
$$
y\le x\quad\Longleftrightarrow\quad y_k\le x_k\quad\text{for every } k=1,\ldots,m.
$$
\et

\begin{proof}
Let $\chi\in C^\infty_b(\mR)$ satisfy
$$
0\leq \chi\leq 1,\ \ \chi(r)=1
\quad\text{for } r\le 0,
\quad
\chi(r)=0
\quad\text{for } r\ge 1.
$$
For fixed $x\in\mR^m$ and $0<\varepsilon\le 1$, define
\begin{equation*}
h_{x,\varepsilon}^+(y)
:=
\prod_{k=1}^m
\chi\left(
\frac{y_k-x_k}{\varepsilon}
\right),\quad
h_{x,\varepsilon}^-(y)
:=
\prod_{k=1}^m
\chi\left(
\frac{y_k-x_k+\varepsilon}{\varepsilon}
\right).
\end{equation*}
Then
\begin{equation*}
h_{x,\varepsilon}^+,h_{x,\varepsilon}^-
\in C_b^\infty(\mR^m)
\subset C_b^2(\mR^m).
\end{equation*}
Moreover, it is easy to see that
\begin{equation*}
\mathbf 1_{\{y\le x\}}
\le
h_{x,\varepsilon}^+(y)
\le
\mathbf 1_{\{y\le x+\varepsilon\}},\quad
\mathbf 1_{\{y\le x-\varepsilon\}}
\le
h_{x,\varepsilon}^-(y)
\le
\mathbf 1_{\{y\le x\}},
\end{equation*}
and for some universal constant $C_\chi>0$,
$$
\|\nabla h_{x,\varepsilon}^{\pm}\|_\infty\leq C_{\chi}m\eps^{-1},\quad 
\|\nabla^2 h_{x,\varepsilon}^{\pm}\|_\infty\leq C_{\chi}m^2\eps^{-2}.
$$
We first obtain an upper bound for the difference of the distribution
functions.
Using \eqref{eq:weak-rate}, we have
\begin{align*}
\mathbb P(Y_t^N\le x)
&\le
\mathbb E h_{x,\varepsilon}^+(Y_t^N)
\leq\mathbb E h_{x,\varepsilon}^+(Y_t)+
|\mathbb E h_{x,\varepsilon}^+(Y_t^N)-\mathbb E h_{x,\varepsilon}^+(Y_t)|
\nonumber\\
&\le
\mathbb E h_{x,\varepsilon}^+(Y_t)
+
C_1 N^{-1/2}(\|\nabla h_{x,\varepsilon}^+\|_{L^\infty}
+\|\nabla^2 h_{x,\varepsilon}^+\|_{L^\infty})
\nonumber\\
&\le
\mathbb P(Y_t\le x+\varepsilon)
+
C_1 C_\chi m^2N^{-1/2}\varepsilon^{-2}.
\end{align*}
Therefore,
\begin{align*}
\mathbb P(Y_t^N\le x)-\mathbb P(Y_t\le x)
&\le
\mathbb P(Y_t\le x+\varepsilon)
-
\mathbb P(Y_t\le x)
+
C_1 C_\chi m^2N^{-1/2}\varepsilon^{-2}.
\end{align*}
Since the $k$-th component $Y_t^{(k)}=\langle\eta_t,\varphi_k\rangle$
is Gaussian with variance $\sigma_{t,k}^2>0$,
we have
\begin{align*}
\mathbb P(Y_t\le x+\varepsilon)
-
\mathbb P(Y_t\le x)
&\le
\sum_{k=1}^m
\mathbb P(x_k<Y_t^{(k)}\le x_k+\varepsilon)
\le
\frac{\varepsilon}{\sqrt{2\pi}}
\sum_{k=1}^m\frac1{\sigma_{t,k}}=:C_2\eps.
\end{align*}
Consequently,
\begin{equation*}
\mathbb P(Y_t^N\le x)-\mathbb P(Y_t\le x)
\le
C_2\varepsilon
+C_1 C_\chi m^2N^{-1/2}\varepsilon^{-2}.
\end{equation*}
For the opposite direction, we use $h_{x,\varepsilon}^-$. First,
\begin{align*}
\mathbb P(Y_t\le x)
&\le
\mathbb P(Y_t\le x-\varepsilon)
+
\mathbb P(Y_t\le x)-\mathbb P(Y_t\le x-\varepsilon)
\le
\mathbb E h_{x,\varepsilon}^-(Y_t)
+C_2\varepsilon.
\end{align*}
By \eqref{eq:weak-rate} again, we also have
\begin{align*}
\mathbb E h_{x,\varepsilon}^-(Y_t)
&\le
\mathbb E h_{x,\varepsilon}^-(Y_t^N)
+
    C_1 N^{-1/2}\bigl(\|\nabla h_{x,\varepsilon}^-\|_{L^\infty}
    +\|\nabla^2 h_{x,\varepsilon}^-\|_{L^\infty}\bigr)
\nonumber\\
&\le
\mathbb P(Y_t^N\le x)
+
C_1 C_\chi m^2N^{-1/2}\varepsilon^{-2}.
\end{align*}
Hence
\begin{equation*}
\mathbb P(Y_t\le x)-\mathbb P(Y_t^N\le x)
\le
C_2\varepsilon
+
C_1 C_\chi m^2N^{-1/2}\varepsilon^{-2}.
\end{equation*}
Combining the two estimates gives
\begin{equation*}
\sup_{x\in\mR^m}
\left|
\mathbb P(Y_t^N\le x)-\mathbb P(Y_t\le x)
\right|
\le
C_3
\left(
\varepsilon
+
m^2N^{-1/2}\varepsilon^{-2}
\right).
\end{equation*}
Finally, taking $\eps=N^{-1/6}$ yields \eqref{Berry-Esseen-m}.
\end{proof}

\vspace{2mm}

\section*{\bf Acknowledgements}
The authors would like to thank Zhenfu Wang for his quite useful discussions on sub-Gaussian
concentration estimates, their $L^p$ extensions, and the corresponding
U-statistic estimates.
Zimo Hao thanks Lucio Galeati for informing him in Marseille about related
work in progress with Avi Mayorcas and Johanna Weinberger.
This work is supported by the NNSF of China (No. 12595282) and
the National Key R\&D Program of China (No. 2023YFA1010103).

\newcommand{\etalchar}[1]{$^{#1}$}

\end{document}